\documentclass[12pt,oneside]{amsart}

\usepackage[utf8]{inputenc}
\usepackage[english]{babel}
\usepackage[margin=3cm]{geometry} 
\usepackage{amsmath,amsthm,amssymb}
\usepackage{mathtools}
\usepackage{graphicx}
\usepackage[colorlinks=true, allcolors=blue]{hyperref}
\usepackage{mathrsfs,calligra}
\usepackage{tikz-cd,adjustbox}
\usepackage{spverbatim,comment}
\usepackage[shortlabels]{enumitem}

\usepackage[alphabetic,initials]{amsrefs}

\theoremstyle{plain}
\newtheorem{theorem}{Theorem}[section]
\newtheorem{lemma}[theorem]{Lemma}
\newtheorem{proposition}[theorem]{Proposition}
\newtheorem{corollary}[theorem]{Corollary}

\newtheorem{thmx}{Theorem}

\newtheorem{corollaryx}[thmx]{Corollary}

\newtheorem{propositionx}[thmx]{Proposition}

\theoremstyle{definition}
\newtheorem{definition}[theorem]{Definition}
\newtheorem{example}[theorem]{Example}

\theoremstyle{remark}
\newtheorem{remark}[theorem]{Remark}

\newtheorem{question}[theorem]{Question}

\newcommand{\Rom}[1]{\uppercase\expandafter{\romannumeral#1}}

\newcommand{\A}{\mathcal A}

\newcommand{\E}{\mathcal E}

\renewcommand{\H}{\mathcal H}

\renewcommand{\O}{\mathcal O}

\newcommand{\bL}{{\bf L}}
\newcommand{\bR}{{\bf R}}

\renewcommand{\AA}{\mathbb A}
\newcommand{\CC}{\mathbb C}
\newcommand{\DD}{\mathbb D}

\newcommand{\HH}{\mathbb H}
\newcommand{\LL}{\mathbb L}

\newcommand{\PP}{\mathbb P}

\newcommand{\ZZ}{\mathbb Z}

\DeclareMathOperator{\PGL}{PGL}

\newcommand{\xto}{\xrightarrow} 
\newcommand{\blank}{{-}} 

\DeclareMathOperator{\sHom}{\mathscr{H}\text{\kern -3pt {\calligra\large om}}\,}
\newcommand{\RHom}{{\bf R} \H om}

\newcommand{\sExt}{\E xt} 

\newcommand{\DB}{\underline{\Omega}} 
\newcommand{\dt}{\otimes^{\LL}} 
\DeclareMathOperator{\qis}{qis}

\DeclareMathOperator{\Cone}{Cone}

\DeclareMathOperator{\cotors}{cotors}

\DeclareMathOperator{\sing}{sing}
\DeclareMathOperator{\supp}{supp}
\DeclareMathOperator{\Spec}{Spec}

\DeclareMathOperator{\codim}{codim}

\DeclarePairedDelimiter\abs{\lvert}{\rvert}

\newcommand{\bracket}[1]{\left(#1\right)}

\makeatletter
\let\oldabs\abs
\def\abs{\@ifstar{\oldabs}{\oldabs*}}
\makeatother

\theoremstyle{definition} 
\newcommand{\thistheoremname}{}
\newtheorem*{genericthm}{\thistheoremname}

\newcommand{\todo}[1]{\textcolor{red}{Todo: #1}}

\newcommand{\Addresses}{{
  \vspace{\bigskipamount}
  \footnotesize
  \textsc{Department of Mathematics, Harvard University, 1 Oxford Street, Cambridge, MA 02138, USA}\par\nopagebreak
  \textit{E-mail address}: \texttt{wshen@math.harvard.edu}

  \vspace{\bigskipamount}

  \textsc{Department of Mathematics, University of Michigan, 530 Church Street, Ann Arbor, MI 48109, USA}\par\nopagebreak
  \textit{E-mail address}: 
  \texttt{srivenk@umich.edu}
  
  \vspace{\bigskipamount}
  
  \textsc{Department of Mathematics, Harvard University, 1 Oxford Street, Cambridge, MA 02138, USA}\par\nopagebreak
  \textit{E-mail address}: \texttt{ducvo@math.harvard.edu}
}}

\title{On $k$-Du Bois and $k$-rational singularities}
\author{Wanchun Shen}
\author{Sridhar Venkatesh}
\thanks{S.V. was partially supported by NSF grant DMS-2001132.}
\author{Anh Duc Vo}

\begin{document}

\maketitle

\begin{abstract}
We introduce new notions of $k$-Du Bois and $k$-rational singularities, extending the previous definitions in the case of local complete intersections (lci), to include natural examples outside of this setting. We study the stability of these notions under general hyperplane sections and show that varieties with $k$-rational singularities are $k$-Du Bois, extending previous results in \cite{MP-lci} and \cite{FL-Saito} in the lci and the isolated singularities cases. In the process, we identify the aspects of the theory that depend only on the vanishing of higher cohomologies of Du Bois complexes (or related objects), and not on the behaviour of the K\"ahler differentials.

\end{abstract}

\section{Introduction}
Rational and Du Bois singularities have been studied extensively because of their close relations to the minimal model program; in particular, Kawamata log terminal singularities are rational, while log canonical singularities are Du Bois. The recent development of Hodge theoretic methods has led to considerable interest in the study of their higher analogues, the $k$-Du Bois and $k$-rational singularities (see \cite{MOPW}, \cite{JKSY}, \cite{FL-Saito}, \cite{FL-isolated}, \cite{MP-lci}, \cite{CDM-k-rational}). A local complete intersection (lci) subvariety $X$ of a smooth variety $Y$ is said to have \textit{$k$-Du Bois singularities} if the canonical morphisms 
\begin{equation}
\Omega^p_X \to \DB^p_X \text{ are isomorphisms for all }0\le p\le k, \tag{strict-$k$-DB}\label{def:old-k-DB}
\end{equation}
where $\DB^p_X$ is the $p$-th graded piece, suitably shifted, of the Du Bois complex of $X$ with respect to the natural filtration. On the other hand, $X$ is said to have \textit{$k$-rational singularities}
if the canonical morphisms 
\begin{equation}
\Omega^p_X \to \DD_X(\DB^{n-p}_X)  
\text{ are isomorphisms for all }0\le p\le k, \tag{strict-$k$-ratl} \label{def:old-k-rat}
\end{equation}
where $\DD_X(-)\coloneqq \RHom(-, \omega_X^\bullet)[-n]$. By \cite[Theorem 3.8, Corollary 3.17, 5.9]{FL-Saito}, this is the same as requiring that the natural maps
\[ \Omega^p_X \to \bR f_*\Omega^p_{\tilde X}(\log E) \text{ are isomorphisms for all }0\le p\le k, \]
for any strong log resolution $f: \tilde X \to X$ with reduced exceptional divisor $E$.  

Although the above definitions make sense for arbitrary varieties, most of the results in \textit{loc. cit.} apply only for hypersurfaces, or more generally, for lci varieties. Outside the lci case, not much is known. In fact, even varieties with very mild singularities, such as cones over Veronese varieties (see Example \ref{ex: cone-over-Pr}), do not satisfy the above condition (\ref{def:old-k-DB}) for any $k \geq 1$, due to the non-reflexivity of the K\"ahler differentials (\cite[Proposition 10]{greb-rollenske-torsion-kahler}). On the other hand, the applications of $k$-Du Bois and $k$-rational singularities to deformation theory raise the need to extend results like \cite[Corollary 3.6, 3.7]{FL-def} to such cones, even when they are not necessarily lci. Thus, to develop a good theory in general, it is desirable to introduce new definitions that agree with (\ref{def:old-k-DB}) and (\ref{def:old-k-rat}) in the lci case, but are flexible enough to include natural non-lci examples. In the present paper, we place our attention on the study of these new notions, and refer to the naive extension of conditions (\ref{def:old-k-DB}) and (\ref{def:old-k-rat}) for non-lci varieties as \textit{strict-$k$-Du Bois} and \textit{strict-$k$-rational singularities}. 

As a first step towards the general theory, we define and study the weaker notions of \textit{pre-$k$-Du Bois} and \textit{pre-$k$-rational} singularities, removing the conditions in cohomological degree $0$.

\begin{definition}\label{definition:pre-k-DB-rational}
Let $X$ be a variety of dimension $n$.
\begin{itemize}
    \item We say that $X$ has \textit{pre-$k$-Du Bois} singularities if 
    \[ \H^i\DB_X^p=0 \]
    for all $i>0$ and $0\le p\le k$.
    \item We say that $X$ has \textit{pre-$k$-rational} singularities if
    \[\H^i(\DD_X(\DB^{n-p}_X))=0\]
    for all $i>0$ and $0\le p\le k$\footnote{When $k<\codim_X (X_{\sing})$, one can show (Lemma \ref{lemma:h0-dual-and-log-differentials}) that this condition is equivalent to the more familiar condition $R^if_*\Omega_{\tilde{X}}^p(\log E)=0$ for all $i>0$ and $p\le k$: see Remark \ref{remark: equiv-def-pre-k-rat}.}.
\end{itemize}
\end{definition} 
The consideration of these notions has the advantage of isolating the problems caused by the non-reflexivity of the K\"ahler differentials, and making clear which aspects of the theory continue to hold without assumptions in cohomological degree $0$. Examples of varieties with pre-$k$-Du Bois or pre-$k$-rational singularities are discussed at the end of the introduction, and in Sections \ref{section:toric-varieties} and \ref{section:cones}.


It was proved in \cite{FL-Saito} that the notions of $k$-Du Bois and $k$-rational singularities are stable under general hyperplane sections for lci varieties. In this paper, we prove that this is also the case for pre-$k$-Du Bois and pre-$k$-rational singularities in the general setting. 

\begin{thmx}\label{theorem:gen-hyperplane-sections-pre-k-DB}
Let $X$ be a quasi-projective variety. If $X$ has pre-$k$-Du Bois (resp. pre-$k$-rational) singularities, then a general hyperplane section $H$ of $X$ also has pre-$k$-Du Bois (resp. pre-$k$-rational) singularities.
\end{thmx}

As a corollary, one can show that if a variety $X$ has strict $k$-rational or strict $k$-Du Bois singularities, then so does a general hyperplane section of $X$.

It is well-known by the works of Kov\'acs \cite{kovacs-rational-singularities-are-DB} and Saito \cite{Saito-MHC} that rational singularities are Du Bois. Recently, it was proved that strict $k$-rational implies strict $k$-Du Bois when $X$ is lci (\cite[Theorem B]{MP-lci}, \cite[Theorem 1.6]{FL-Saito}), or when $X$ has isolated singularities (\cite[Theorem 3.20]{FL-Saito}). Along these lines, we prove:

\begin{thmx}\label{theorem:pre-k-rational-implies-pre-k-DB}
Let $X$ be a normal variety with pre-$k$-rational singularities. Then $X$ has pre-$k$-Du Bois singularities.
\end{thmx}
 
As a corollary, we obtain 

\begin{corollaryx} \label{corollary:k-rational-implies-k-DB}
Let $X$ be a variety with strict $k$-rational singularities. Then $X$ has strict $k$-Du Bois singularities.  
\end{corollaryx}

Using Theorem \ref{theorem:pre-k-rational-implies-pre-k-DB}, one can show that the Hodge-Du Bois numbers of a normal variety with pre-$k$-rational singularities enjoy partial Hodge symmetry (Corollary \ref{cor:Hodge-symmetry}). Moreover, the same techniques of the proof of Theorem \ref{theorem:pre-k-rational-implies-pre-k-DB} can be used to study the behavior of pre-$k$-Du Bois singularities under finite maps (Proposition \ref{proposition-boutot-pre-k-DB}). Using this, we obtain a hyperresolution-free proof of the classical result \cite[Theorem 5.3]{DB} that describes the Du Bois complex for (finite) quotient singularities (Corollary \ref{cor:quotient-singularities-pre-n-DB}). 
 
We now proceed to discuss notions that take into account the behavior of the $0$-th cohomology sheaves. Despite the validity of the results of Corollary \ref{corollary:k-rational-implies-k-DB}, we stress that the notions of strict $k$-Du Bois and strict $k$-rational singularities are too restrictive for non-lci varieties, and need not be the right definitions of $k$-Du Bois and $k$-rational singularities in general. We propose instead the following new definitions. 

\begin{definition} 
Let $X$ be a variety. $X$ has \textbf{$k$-Du Bois singularities} if it is seminormal, and 
\begin{enumerate}
\item $\codim_X (X_{\sing}) \ge 2k+1$;
\item $X$ has pre-$k$-Du Bois singularities;
\item $\H^0\DB_X^p$ is reflexive, for all $p \le k$.
\end{enumerate}
\end{definition}

\begin{definition}
Let $X$ be a variety. $X$ is said to have \textbf{$k$-rational singularities} if it is normal, and
\begin{enumerate}
\item $\codim_X (X_{\sing}) > 2k+1$;
\item For all $i>0$ and $0\le p\le k$, 
\[ R^if_*\Omega^p_X(\log E) = 0\]
for any strong log resolution $f:\tilde X \to X$. 
\end{enumerate}
\end{definition}
\noindent By \cite[Lemma 3.14]{FL-Saito}, condition $(2)$ can be replaced by:


\begin{flushleft}
\hspace{5mm}   $(2')$ $X$ has pre-$k$-rational singularities.
\end{flushleft}
One important feature of $k$-rational and $k$-Du Bois singularities in the lci case is that they imply bounds on the codimension of the singular locus (\cite[Corollary 9.26]{MP-local-cohomology}, \cite[Corollary 1.3]{CDM-k-rational}), which are useful in applications, (see e.g. Remark \ref{remark:new-notions}). In general, such codimension bounds are not guaranteed, and so we include them as part of the definitions. One verifies that in the lci case, the above definitions agree with the well-studied notions of strict $k$-Du Bois and strict $k$-rational singularities (Proposition \ref{proposition: def-equiv-lci}), and extend the classical notions of Du Bois and rational singularities when $k=0$ (Proposition \ref{proposition:k-DB-rational-0}). Moreover, the stability under general hyperplane sections, and the implication $k$-rational $\implies$ $k$-Du Bois continue to hold. 
\begin{thmx}\label{corollary:new-k-rat-implies-new-k-DB}
Let $X$ be a quasi-projective variety. We have
\begin{enumerate}[label = (\alph*)]
    \item If $X$ has rational and $k$-Du Bois (resp. $k$-rational) singularities, then so does a general hyperplane section $H$ of $X$.
    \item If $X$ has $k$-rational singularities, then it has $k$-Du Bois singularities.
\end{enumerate}
\end{thmx}

Significantly, there is an abundance of varieties with $k$-Du Bois or $k$-rational singularities as defined above. We characterize these notions for toric varieties and cones over smooth projective varieties, which are typically non-lci (for more concrete examples of cones over Veronese varieties and K3 surfaces, see Example \ref{ex:cone-over-Veronese-and-K3}).

\begin{propositionx}\label{criteria:toric-varieties}
Let $X$ be an affine toric variety and let $c := \codim_X(X_{\sing})$. Then:
    \begin{enumerate}
        \item $X$ has pre-$k$-Du Bois singularities for all $k \geq 0$. It has $k$-Du Bois singularities for all $0\le k\le \frac{c-1}{2}$.
        \item If $X$ is simplicial, then it has pre-$k$-rational singularities for all $k \geq 0$. It has $k$-rational singularities for all $0\le k<\frac{c-1}{2}$.
        \item If $X$ is non-simplicial, then it does not have pre-$1$-rational singularities, hence it does not have $1$-rational singularities.
    \end{enumerate}
\end{propositionx}

\begin{propositionx}
Let $X$ be a smooth projective variety of dimension $n$, and $L$ an ample line bundle on $X$. Let \[C(X,L)=\Spec \bigoplus_{m\ge 0} H^0(X,L^m)\]
be the affine cone over $X$ with conormal bundle $L$. Then:
\begin{enumerate}
\item $C(X,L)$ is pre-$k$-Du Bois if and only if 
\[H^i(X,\Omega^p_X\otimes L^m)=0 \text{ for all } i>0, p\le k, m\ge 1.\]
\item $C(X,L)$ is $k$-Du Bois if and only if it is pre-$k$-Du Bois, $k\le n/2$ and
\[H^i(X,\O_X)=0 \text{ for all } 0<i\le k.\] 
\item For $k \leq n$, the cone $C(X,L)$ is pre-$k$-rational if and only if
\[H^i(X,\Omega^p_X\otimes L^m)=0 \text{ for all } i>0, p\le k, m\ge 0,\]
except possibly when $m=0, i=p$, in which case
\[H^0(X,\O_X)\overset{\simeq}{\longrightarrow} H^1(X,\Omega_X^1)\overset{\simeq}{\longrightarrow}\cdots \overset{\simeq}{\longrightarrow} H^k(X,\Omega_X^k)\]
are isomorphisms via the map $c_1(L)$.
If $C(X,L)$ is pre-$n$-rational, then it is pre-$(n+1)$-rational.
\item $C(X,L)$ is $k$-rational if it is pre-$k$-rational and $k<n/2$.
\end{enumerate}
\end{propositionx}
More examples of $k$-Du Bois and $k$-rational singularities in this sense are given by secant varieties, see \cite{secant}.

Note that, unlike in the lci case, it need not be true that $k$-Du Bois singularities are $(k-1)$-rational (see Remark \ref{remark:-k-Du-Bois-but-not-k-1-rational}). In view of such phenomena, it is possible that these notions may require further small adjustments; nonetheless, we intend them as the basis for the development of a general theory.


\textbf{Outline of the paper.} In Section \ref{sec:preliminaries}, we review the construction of the Du Bois complex, and list some well-known facts about it, as well as other preliminaries for the subsequent discussion. In Section \ref{section:hyperplane-cutting-argument}, we discuss the behavior of pre-$k$-Du Bois and pre-$k$-rational singularities under general hyperplane sections. In Section \ref{section:k-rational-singularity-and-k-Du-Bois}, we prove our main theorem that pre-$k$-rational (resp. strict $k$-rational) singularities are pre-$k$-Du Bois (resp. strict $k$-Du Bois). In Section \ref{section:definitions}, we propose new definitions of $k$-Du Bois and $k$-rational singularities in general, and explain how the results in previous sections carry over without much difficulty. In the final two sections \ref{section:toric-varieties} and \ref{section:cones}, we study the examples of toric varieties and cones over smooth projective varieties. 

\textbf{Acknowledgements.} We would like to express our sincere gratitude to Mircea Musta\c{t}\u{a} and Mihnea Popa for their constant support during the preparation of this paper, and in particular for suggestions regarding the definitions in section \ref{section:definitions}. We are grateful to Sándor Kovács and Radu Laza for their insightful comments on our draft. We also thank Joaqu\'in Moraga, Sung Gi Park, Benjamin Tighe and Jakub Witaszek for helpful discussions.

\section{Preliminaries}\label{sec:preliminaries}
Throughout the paper, by a variety we mean a reduced, separated scheme of finite type over $\CC$ (not necessarily irreducible). A resolution of singularities $f:\tilde X \to X$ is said to be a \textit{strong log resolution} with exceptional divisor $E$ if it is an isomorphism away from the singular locus $X_{\sing}$, and $E := f^{-1}(X_{\sing})_{\text{red}}$ is a simple normal crossing divisor. 
\subsection{The Du Bois complex}
Let us recall the definition of the Du Bois complex. For a more detailed treatment, we refer to \cite[Section 7.3]{peters-steenbrink} (see also \cite{DB}, \cite{GNPP}).

Let $X$ be a variety of dimension $n$. Following the idea of Deligne, Du Bois introduced in \cite{DB} the \textit{filtered de Rham complex}, or the \textit{Du Bois complex} $(\DB_X^{\bullet},F)$, which is an object in the bounded derived category of filtered differential complexes. The graded pieces with respect to the filtration $F$
\[\DB^p_X \coloneqq gr^p_F \DB^\bullet_X[p]\]
are objects in $D^b_{\rm coh}(X)$. It can be constructed explicitly as follows: 
\begin{enumerate}
\item Construct a simplicial resolution $\epsilon: X_{\bullet}\to X$ of $X$ using either cubical \cite{GNPP} or polyhedral \cite{Carlson-poly-res} resolutions. The simplicial resolution constructed using the method \textit{loc.cit} has the property that $\dim X_i \le n-i$. In particular, it has length $\le n$. 
\item The $p$-th graded piece $\DB_X^p:=\bR \epsilon_{\bullet *}\Omega^p_{X_{\bullet}}$ of the Du Bois complex can be represented by the single complex associated to the double complex
\[\begin{tikzcd}
&\, &0 &0 &\, &0 &\, \\
&0\ar[r] &\epsilon_*\A^{p,0}_{X_n} \ar[u]\ar[r]&\epsilon_*\A^{p,1}_{X_n}\ar[u]\ar[r] &\cdots \ar[r] &\epsilon_*\A^{p,n-p}_{X_n}\ar[r]\ar[u] &0 \\
&\, &\vdots \ar[u]&\vdots\ar[u] &\, &\vdots\ar[u] &\, \\
&0\ar[r] &\epsilon_*\A^{p,0}_{X_1} \ar[u]\ar[r]&\epsilon_*\A^{p,1}_{X_1}\ar[u]\ar[r] &\cdots \ar[r] &\epsilon_*\A^{p,n-p}_{X_1}\ar[r]\ar[u] &0 \\
&0\ar[r] &\epsilon_*\A^{p,0}_{X_0} \ar[u,"\epsilon_0-\epsilon_1"]\ar[r,"\overline{\partial}"]&\epsilon_*\A^{p,1}_{X_0}\ar[u]\ar[r] &\cdots \ar[r] &\epsilon_*\A^{p,n-p}_{X_0}\ar[r]\ar[u] &0 \\
&\, &0\ar[u] &0\ar[u] &\, &0\ar[u]
\end{tikzcd}
\tag{$\star$}\label{double-complex}\]
where $\A^{p,q}_{X_i}$ is the sheaf of smooth $(p,q)$-forms on $X_i$. The horizontal differentials in the double complex are given by $\overline{\partial}$, and the vertical differentials are induced by alternating sums of the pullback maps from the simplicial resolution.
\end{enumerate}
Since rows of (\ref{double-complex}) represent the objects $\bR \epsilon_{i*} \Omega^p_{X_i} \in D^b_{\rm coh}(X_i)$, it follows that taking the total complex associated to (\ref{double-complex}) amounts to successively taking cones of the morphisms $\bR \epsilon_{i*} \Omega^p_{X_i} \to \bR \epsilon_{i+1*} \Omega^p_{X_{i+1}}$. More precisely, we have:
\begin{lemma}\label{lemma:DB-cpx-out-of-cones}
Let $X$ be a variety, and $\epsilon: X_{\bullet}\to X$ a simplicial resolution. Then the p-th graded piece of the Du Bois complex is ``built out of cones":
\[\DB_X^p=\Cone(\dots (\Cone(\bR\epsilon_{0*}\Omega_{X_0}^p,\bR\epsilon_{1*}\Omega_{X_1}^p),\bR\epsilon_{2*}\Omega_{X_2}^p), \dots,\bR\epsilon_{m*}\Omega_{X_m}^p)[-m].\]
\end{lemma}
\noindent This observation will be useful in our proof of the base change formula (Lemma \ref{lemma:conormal-DB-SES}).


We summarize other well-known facts about Du Bois complexes that will be useful later. For a more comprehensive list of basic properties of $\DB_X^{\bullet}$, we refer the readers to \cite[Theorem 4.2]{kovacs-schwede-survey}.
\begin{proposition}\label{proposition:facts-about-DB-complex}
Let $X$ be a variety, and $(\DB_X^{\bullet},F)$ the Du Bois complex on $X$. We denote $\DB_X^{\le p}:= \DB_X^{\bullet}/F^p\DB_X^{\bullet}$.
\begin{enumerate}
\item There is a spectral sequence
\[E_1^{q,p}=\HH^q(X,\DB_X^p)\implies H^{p+q}(X,\CC).\]
When $X$ is proper, it degenerates at $E_1$, so the Hodge filtration on $H^{p+q}(X,\CC)$ is given by
\[F^pH^{p+q}(X,\CC) = \HH^q(X,F^p\DB_X^{\bullet}).\]
In particular, there is a surjection
\[H^i(X,\CC) \twoheadrightarrow H^i(X,\CC)/F^pH^i(X,\CC) = \HH^i(X,\DB_X^{\le p})\]
for every $i$ and $p$.

\item Let $Z$ be a closed reduced subscheme of $X$. Let $f: \tilde X \to X$ be a resolution of singularities which is an isomorphism outside $Z$, and $E := f^{-1}(Z)_{\mathrm{red}}$ is a simple normal crossing divisor. By \cite[Example 7.25]{peters-steenbrink}, we have the following exact triangle for all $p$:
\begin{equation}
\notag\label{equation:steenbrink-exact-triangle}
    \bR f_*\Omega^p_{\tilde X}(\log E)(-E) \to \DB^p_X \to \DB^p_Z \xto{+1}.
\end{equation}
\item Suppose $X$ is quasi-projective, and $H$ a general hyperplane section of $X$. Let $\epsilon'_\bullet: H_\bullet := \epsilon_\bullet^{-1}(H) \to H$ be the induced map via pullback. Since $H$ is general, $\epsilon'_\bullet$ is also a simplicial resolution of $H$. This can be summarized in the following diagram:
\[\begin{tikzcd}    &H_\bullet\ar[r,"i_\bullet"]\ar[d,swap,"\epsilon_\bullet'"] &X_\bullet \ar[d,"\epsilon_\bullet"] \\
    &H \ar[r,"i"] &X.
\end{tikzcd}\]
\end{enumerate}
\end{proposition}


Analogous to the differentials of the de Rham complex, we have morphisms
\[d: \DB^{j}_X \to \DB^{j+1}_X,\] which can be defined as the connecting morphism of the exact triangle
\[F^{j+1}\DB_X^\bullet/F^{j+2}\DB_X^\bullet \to F^{j}\DB_X^\bullet/ F^{j+2}\DB_X^\bullet \to F^{j}\DB_X^\bullet/F^{j+1}\DB_X^\bullet \xto{+1}.\]
The degree zero pieces give rise to a complex of sheaves, which we denote by
\[\Omega_{X,h}^{\le p}:=[\H^0\DB^0_X \xto{d} \H^0\DB^1_X \xto{d} \dots \xto{d} \H^0\DB^p_X],\]
placed in cohomological degrees $0,1,\dots, p$. Our choice of notation is motivated by the fact that $\H^0\DB_X^p$ agrees with the $h$-differentials studied in \cite{HJ14}.

\begin{proposition}\label{prop:Hodge-theory-surjection}
Let $X$ be a proper variety. Then there is a surjection
\[\HH^i(X,\Omega_{X,h}^{\le p})\to \HH^i(X,\DB_X^{\le p})\]
for all $i$ and $p$.
\begin{proof}
We first argue by induction on $p$ that there are natural maps
\[\Omega_{X,h}^{\le p}\to \DB_X^{\le p}.\]
When $p=0$, it is clear that we have a map $\Omega_{X,h}^{\le 0}=\H^0\DB_X^0 \to \DB_X^0$. Assume we have defined a map $\Omega_{X,h}^{\le p-1}\to \DB_X^{\le p-1}$, then from the following exact triangles
\[\begin{tikzcd}
\H^0\DB^p_X[-p] \ar[r]\ar[d] &\Omega_{X,h}^{\le p}\ar[r]\ar[d, dotted] &\Omega^{\le p-1}_{X,h} \ar[r,"+1"]\ar[d] &\, \\
\DB_X^{p}[-p] \ar[r] &\DB_X^{\le p}\ar[r]&\DB_X^{\le p-1}\ar[r,"+1"] &\, \\
\end{tikzcd}\]
we obtain a map $\Omega_{X,h}^{\le p}\to \DB_X^{\le p}$.

The natural maps $\CC \to \Omega_{X,h}^{\le p} \to \DB^{\le p}_X$ gives rise to a commutative diagram
\[\begin{tikzcd}
&H^i(X,\CC)\ar[rr]\ar[rd] & &\HH^i(X,\DB_X^{\le p}) \\
&\, &\HH^i(X,\Omega_{X,h}^{\le p})\ar[ru] &\,
\end{tikzcd}\]
For each $i$ and $p$, the horizontal map is surjective by Proposition \ref{proposition:facts-about-DB-complex}(1). It follows that
\[\HH^i(X,\Omega_{X,h}^{\le p})\to \HH^i(X,\DB_X^{\le p})\]
is surjective.
\end{proof}
\end{proposition}

\subsection{Grothendieck duality}
Following notations in \cite[Section 3.3]{FL-Saito}, we consider the shifted Grothendieck duality functor on an $n$-dimensional variety $X$ as
\[\DD_X(\blank):= \RHom(\blank, \omega_X^{\bullet})[-n].\]
It is straight-forward to check that $\DD_X(\DB_X^p)$ is supported in non-negative degree, and that if $Z\subset X$ is a subvariety of codimension $c$, then
\[\DD_X(\blank) = \DD_Z(\blank)[-c].\]

For all $p$, dualizing the exact triangle in Proposition \ref{proposition:facts-about-DB-complex}(2), we obtain the exact triangle:
    \[ \DD_X(\DB^{n-p}_Z) \to \DD_X(\DB^{n-p}_X) \to \bR f_*\Omega^p_{\tilde X}(\log E)\xto{+1}. \]

\begin{lemma}\label{lemma:h0-dual-and-log-differentials}
Let $X$ be a variety.
\begin{enumerate}
    \item $\DD_X(\DB^{n-p}_X) \to \bR f_*\Omega^p_{\tilde X}(\log E)$ is an isomorphism for $p < \codim_X(X_{\sing})$.
    \item  If $X$ is normal, then the natural map
\[ \H^0(\DD_X(\DB^{n-p}_X)) \to f_*\Omega^p_{\tilde X}(\log E). \]
is an isomorphism for all $p$.
\end{enumerate} 
\begin{proof}
Let $Z$ denote $X_{\sing}$. The first statement is clear because $\DB_Z^{n-p}=0$ for $p<\codim_X Z$. For the second statement, since $X$ is normal, $c:=\codim_X Z\ge 2$, so $\DD_X(\DB^{n-p}_Z) \cong \DD_Z(\DB^{n-p}_Z)[-c]$ is supported in degree $\ge 2$; it then follows from the above exact triangle that    \[\H^0(\DD_X(\DB^{n-p}_X)) \to f_*\Omega^p_{\tilde X}(\log E)\]
is an isomorphism for all $p$.
\end{proof}
\end{lemma}

\subsection{Reflexive differentials}
Let $X$ be a variety. We denote the reflexive differential on $X$ by
\[\Omega_X^{[p]}:=(\Omega^p_X)^{\vee\vee},\]
where $(\blank)^{\vee}=\mathcal{H}om(\blank,\O_X)$. 

When $X$ is normal, $\Omega_X^{[p]} \cong j_*\Omega_U^p$, where $j: U\to X$ is the inclusion of the smooth locus $U$ into $X$. 

When $X$ has rational singularities, deep theorems of Kebekus-Schnell \cite{Kebekus-Schnell} imply that the sheaves $\H^0\DB_X^p, f_*\Omega^p_{\tilde{X}}$ and $\H^0(\DD_X(\DB_X^{n-p})) \cong f_*\Omega^p_{\tilde X}(\log E)$ (where $f: \tilde{X}\to X$ is a strong log resolution of $X$ with reduced exceptional divisor $E$) are all isomorphic to the reflexive differential $\Omega^{[p]}_X$. The following remark gives more precise information about the isomorphisms between these sheaves, which will be useful later for technical purposes.
\begin{remark}\label{remark:comparing-various-differentials}
Let $X$ be a variety with rational singularities. Then we have a commutative diagram
    \[
    \begin{tikzcd}
        \Omega^p_X \ar[r,"a"] & \H^0\DB^p_X \ar[d, swap, "g"] \ar[r, "b"] & \H^0(\DD_X(\DB_X^{n-p})) \ar[d, "\simeq"]\\ 
        & \Omega_X^{[p]} & f_*\Omega^p_{\tilde X}(\log E)  \ar[l, "h"]
    \end{tikzcd}
    \]
    Combining \cite[Theorem 7.12]{HJ14} and \cite[Corollary 1.12]{Kebekus-Schnell}, we have that $g$ is an isomorphism. By \cite[Corollary 1.8]{Kebekus-Schnell}, $h$ is an isomorphism. As a consequence, we see that $b \circ a$ is an isomorphism if and only if $a$ is an isomorphism.
\end{remark}

\section{Behavior under taking general hyperplane sections} \label{section:hyperplane-cutting-argument}

The goal of this section is to establish Theorem \ref{theorem:gen-hyperplane-sections-pre-k-DB}.
The technical ingredient in the proof is the following base change result for hyperresolutions.

\begin{lemma}\label{lemma:DB-of-hyperplane}
Let $X$ be a quasi-projective variety, and $H\xhookrightarrow{i} X$ a general hyperplane section. Then for every $p$, we have
    \[ \bR\epsilon'_{\bullet*}(\Omega_{X_\bullet}^p \dt_{\mathcal{O}_{X_\bullet}} \O_{H_\bullet}) \simeq \bR\epsilon_{\bullet*}\Omega_{X_{\bullet}}^p \dt_{\mathcal{O}_X} \O_H,\]
    where $\epsilon: X_{\bullet} \to X$ is any simplicial resolution of $X$, and $\epsilon': H_{\bullet}=\epsilon_{\bullet}^{-1}(H) \to H$ the corresponding simplicial resolution of $H$ as in Proposition \ref{proposition:facts-about-DB-complex}(3).
\end{lemma}
\begin{proof}
By Lemma \ref{lemma:DB-cpx-out-of-cones}, we get
\[ \bR\epsilon_{\bullet*}\Omega_{X^{\bullet}}^p = \Cone(\dots (\Cone(\bR\epsilon_{0*}\Omega_{X_0}^p,\bR\epsilon_{1*}\Omega_{X_1}^p),\bR\epsilon_{2*}\Omega_{X_2}^p), \dots,\bR\epsilon_{m*}\Omega_{X_m}^p)[-m].\]
Being an exact functor, $\bL i^*$ commutes with taking cones, so
\begin{align*}
    \bL i^*\bR\epsilon_{\bullet*}\Omega_{X^{\bullet}}^p &= \bL i^*\Cone(\dots (\Cone(\bR\epsilon_{0*}\Omega_{X_0}^p,\bR\epsilon_{1*}\Omega_{X_1}^p), \dots,\bR\epsilon_{m*}\Omega_{X_m}^p)[-m]\\
    &= \Cone(\dots (\Cone(\bL i^*\bR\epsilon_{0*}\Omega_{X_0}^p,\bL i^*\bR\epsilon_{1*}\Omega_{X_1}^p), \dots,\bL i^*\bR\epsilon_{m*}\Omega_{X_m}^p)[-m]\\
    &= \Cone(\dots (\Cone(\bR\epsilon'_{0*}\bL i_0^*\Omega_{X_0}^p,\bR\epsilon'_{1*}\bL i_1^*\Omega_{X_1}^p), \dots,\bR\epsilon'_{m*}\bL i_m^*\Omega_{X_m}^p)[-m]\\
    &= \bR\epsilon'_{\bullet*}\bL i_\bullet^*\Omega^p_{X_{\bullet}},
\end{align*}
where we used Lemma \ref{lemma:DB-cpx-out-of-cones} again in the last equality. The third equality holds by the base change formula \cite[\href{https://stacks.math.columbia.edu/tag/08IB}{Tag 08IB}]{stacks-project}, where the Tor-independence of $X_i$ and $H$ over $X$ is satisfied because $H$ is general.
\end{proof}


From this lemma, we obtain an analogue of the conormal exact sequence for the  graded pieces of the Du Bois complex. 

\begin{lemma}\label{lemma:conormal-DB-SES}
Let $X$ be a quasi-projective variety and $H\subset X$ a general hyperplane section. Then there is an exact triangle:
\[\DB_H^{p-1}\dt_{\O_H} \O_H(-H)\to \DB_X^p\dt_{\O_X} \O_H\to \DB_H^p\xto{+1},\]
which we abbreviate as:
\begin{equation}
\notag\label{equation:exact-triangle-DB-hyperplane}
    \DB_H^{p-1}(-H)\to \DB_X^p|_H\to \DB_H^p\xto{+1}.
\end{equation}
\end{lemma}
\begin{proof}
    By \cite[ V.2.2.1]{GNPP}, we have an exact triangle:
    \[ \DB_H^{p-1}\dt_{\O_H} \O_H(-H)\to \bR\epsilon'_{\bullet*}(\Omega_{X_\bullet}^p \dt_{\O_{X_\bullet}} \O_{H_\bullet}) \to \DB_H^p\xto{+1}. \]
    Thus, we only need to show that
    \[ \bR\epsilon'_{\bullet*}(\Omega_{X_\bullet}^p \dt_{O_{X_\bullet}} \O_{H_\bullet}) \simeq \DB_X^p\dt_{\O_{X}} \O_H, \]
    which holds by Lemma \ref{lemma:DB-of-hyperplane}.
\end{proof}

We are now ready to show that pre-$k$-Du Bois and pre-$k$-rational singularities are stable under taking general hyperplanes.

\begin{proof}[Proof of Theorem \ref{theorem:gen-hyperplane-sections-pre-k-DB}]
We first treat the pre-$k$-Du Bois case. We prove the statement by induction on $k$, with the base case $k=-1$ being vacuous. Let $H\subset X$ be a general hyperplane section. By Lemma \ref{lemma:conormal-DB-SES}, we have an exact triangle \[\DB^{p-1}_H(-H)\to \DB^p_X|_H \to \DB_H^p\xto{+1}.\]
For $i>0$, we get 
\[\H^i (\DB^{p-1}_H(-H)) = 0\] 
by the induction hypothesis, and 
\[ \H^i(\DB^p_X|_H) \cong (\H^i\DB^p_X)|_H= 0\] 
because $X$ has pre-$k$-Du Bois singularities and $H$ is general. By taking cohomology of the above exact triangle, it follows that $H$ has pre-$k$-Du Bois singularities.

To show that $H$ has pre-$k$-rational singularities, we dualize the above exact triangle to get:
\[\DD_H(\DB^{n-p}_H) \to \DD_H(\DB^{n-p}_X|_H) \to \DD_H(\DB^{n-p-1}_H(-H)) \xto{+1}.\]
Observe that the middle term is identified with 
\[ \DD_H(\DB^{n-p}_X|_H) \cong i^!\DD_X(\DB^{n-p}_X)[1] \cong \DD_X(\DB^{n-p}_X)(-H)|_H,\]
by \cite[Tag 0AU3, 4(b) and 7]{stacks-project}. Thus the above exact triangle is the same as
\[\DD_H(\DB^{n-p}_H) \to \DD_X(\DB^{n-p}_X)(-H)|_H\to \DD_H(\DB^{n-p-1}_H)(H) \xto{+1}.\]
Now, the same argument as above shows that $H$ has pre-$k$-rational singularities. 
\end{proof}

\begin{corollary}\label{cor:old-k-DB-stable-under-hyperplanes}
Let $X$ be a quasi-projective variety. If $X$ has strict $k$-Du Bois (resp. strict-$k$-rational) singularities, then a general hyperplane section $H$ of $X$ also has strict-$k$-Du Bois (resp. strict-$k$-rational) singularities.
\begin{proof}
We argue by induction on $k$. Assume $X$ satisfies has strict $k$-Du Bois singularities, then $H$ has pre-$k$-Du Bois singularities by Theorem \ref{theorem:gen-hyperplane-sections-pre-k-DB}. By induction we obtain $\Omega_H^{p}\to \DB_H^p$ is an isomorphism for $p\le k-1$, hence it suffices to show the natural map
\[\Omega_H^k \to \H^0\DB_H^k\]
is an isomorphism. We have a commutative diagram:
\[
\begin{tikzcd}
&\, &\Omega_H^{k-1}(-H) \ar[r]\ar[d] &\Omega_X^k|_H \ar[r]\ar[d] &\Omega_H^k \ar[r]\ar[d] &0 \\
&0\ar[r] &\H^0\DB_H^{k-1}(-H) \ar[r] &\H^0(\DB_X^k|_H)\ar[r] &\H^0\DB_H^k \ar[r] &0
\end{tikzcd}\]
where the second row is obtained by taking cohomology of the exact triangle (\ref{equation:exact-triangle-DB-hyperplane}). The first column is an isomorphism by induction, and the second column is an isomorphism because $X$ has strict $k$-Du Bois singularities. It follows that the third column is an isomorphism, as desired.

The proof for strict $k$-rational singularities is similar.
\end{proof}
\end{corollary}



\section{pre-$k$-rational implies pre-$k$-Du Bois and Corollaries}\label{section:k-rational-singularity-and-k-Du-Bois}

In this section, we give proofs of Theorem \ref{theorem:pre-k-rational-implies-pre-k-DB} and Corollary \ref{corollary:k-rational-implies-k-DB}, following the strategy of Koll\'ar \cite{Kollar} and Kov\'acs \cite{kovacs-rational-singularities-are-DB} (see also \cite{FL-Saito}).

Before starting the proof, we outline an argument when $X$ is projective, which we hope will make the proof in the general case easier to follow. Note that we have natural maps
\[\H^0\DB_X^k \xto{f} \DB_X^k \xto{g} \DD_X(\DB_X^{n-k}).\]
\begin{enumerate}
    \item Since $X$ is normal and pre-$k$-rational, the composition $g\circ f$ is a quasi-isomorphism. It follows that $f$ induces injective map
    \begin{equation}         H^i(X,\H^0\DB_X^k)\to \HH^i(X,\DB_X^k)\tag{$\ast$} \label{eqn:hypercoh} 
    \end{equation}
    on hypercohomology for all $i$; 
    \item Using surjectivity given by Hodge theory (Proposition \ref{prop:Hodge-theory-surjection}), one can show the map (\ref{eqn:hypercoh})
    is surjective when $X$ is pre-$(k-1)$-Du Bois, which we may assume by induction on $k$;
    \item Combining (1) and (2), we see that 
    (\ref{eqn:hypercoh})
    is an isomorphism. This implies \[\H^0\DB_X^k \xto{\qis} \DB_X^k\](i.e. $X$ is pre-$k$-Du Bois), provided the ``bad locus"
    \[\Sigma_X:= \supp \Cone(\H^0\DB_X^k \to \DB_X^k)\]
    is zero-dimensional;
    \item Using induction, as well as results about hyperplane sections developed in Section \ref{section:hyperplane-cutting-argument}, one can reduce to the case $\dim \Sigma_X\le 0$. 
\end{enumerate}
When $X$ is not necessarily projective, Proposition \ref{prop:Hodge-theory-surjection} does not apply directly. The trick, due to Kov\'acs, is to compactify $X$ and use excision of local cohomology to get rid of the contribution from the boundary of the compactification.

\vspace{\medskipamount}

Now we come to the proof in full generality. 
\begin{proof}[Proof of Theorem \ref{theorem:pre-k-rational-implies-pre-k-DB}]
The properties of having pre-$k$-Du Bois and pre-$k$-rational singularities are both local, so we can reduce to the case when $X$ is quasi-projective. 

We proceed by induction on both $\dim X$ and $k$. The theorem is trivial when $X$ has dimension $0$, or $k=-1$. Suppose the theorem holds for varieties of dimension at most $n-1$, and up to $k-1$ for varieties of dimension $n$. Consider an $n$-dimensional variety $X$ with pre-$k$-rational singularities. By the induction hypothesis, it suffices to show $\H^i\DB^k_X =0$ for all $i>0$, or equivalently, the morphism $\H^0 \DB^k_X \to \DB^k_X$ is an isomorphism. Let $\Sigma_X$ be the locus where this morphism is not an isomorphism.

We first show that $\dim \Sigma_X \le 0$. Let $H$ be a general hyperplane section of $X$. By Lemma \ref{lemma:conormal-DB-SES}, we have the exact triangle 
    \[\DB^{k-1}_H(-H) \to {\DB^k_X|_H} \to {\DB^k_H} \xto{+1}.\]
    By Theorem \ref{theorem:gen-hyperplane-sections-pre-k-DB}, $H$ is pre-$k$-rational, hence has pre-$k$-Du Bois singularities by the induction hypothesis. The associated long exact sequence of the above exact triangle gives 
    \[(\H^i\DB^k_X)|_H \cong \H^i(\DB^{k}_X|_H) = 0 \text{ for } i>0.\]
    Thus we have 
    \[\Sigma_X \cap H = \emptyset.\]
    Since $H$ is general, we get $\dim \Sigma_X\le 0$. 

    \vspace{\medskipamount}

In a series of steps, we will now prove that $\Sigma_X = \emptyset$, thus completing the proof.

\vspace{\medskipamount}

\noindent \textbf{Step 1.} Let us first show that it suffices to prove $H^i(X,\H^0\DB_X^k)\xrightarrow{f_i} \HH^i(X,\DB_X^k)$ is an isomorphism for all $i$.

Indeed, consider $C^{\bullet}=\Cone\bracket{\H^0\DB_X^k\to \DB_X^k}$, this is supported on $\Sigma_X$. The above isomorphisms imply that $\HH^i(X,C^{\bullet})=0$ for all $i$. Since $\dim \Sigma_X \leq 0$, the spectral sequence
    \[E_2^{pq} = H^p(X,\H^q C^{\bullet})\implies \HH^{p+q}(X,C^{\bullet})\]
degenerates at $E_2$ and $H^0(X,\H^q C^{\bullet})=\HH^q(C^{\bullet})=0$ for all $q$. Thus, we have $\H^q C^{\bullet}=0$ for all $q$. This shows that
    \[\H^0\DB^k_X \to \DB^k_X.\]
is an isomorphism.

\vspace{\medskipamount}

\noindent\textbf{Step 2.} We now show that to prove $f_i$ is an isomorphism for all $i$, it suffices to prove 
\[ \HH^i_{\Sigma_X}(X,\Omega_{X,h}^{\le k}) \xrightarrow{\psi_i} \HH^i_{\Sigma_X}(X,\DB^{\le k}_X)\]
is an isomorphism for all $i$. We do this because we can use arguments from the projective case when considering $\psi_i$ (see Step 4 for more details).

Consider the diagram
\[    \begin{tikzcd} 
	{\HH^{i-1}(X,\Omega_{X,h}^{\le k-1})) } & {\HH^i(X,\H^0(\DB^k_X)[-k]) } & {\HH^i(X,\Omega_{X,h}^{\le k}))} & {\HH^i(X,\Omega_{X,h}^{\le k-1})) } \\
	{\HH^{i-1}(X, \DB^{\le k-1}_X) } & {\HH^i(X,\DB^k_X[-k])} & {\HH^i(X,\DB^{\le k}_X) } & {\HH^i(X, \DB^{\le k-1}_X) }.
	\arrow[from=1-2, to=1-3]
	\arrow[from=1-3, to=1-4]
	\arrow["{f_i[-k]}", from=1-2, to=2-2]
	\arrow[from=2-2, to=2-3]
	\arrow[from=2-3, to=2-4]
	\arrow["{\phi_i}", from=1-3, to=2-3]
	\arrow["{\cong }"', from=1-4, to=2-4]
	\arrow[from=2-1, to=2-2]
	\arrow[from=1-1, to=1-2]
	\arrow["\cong"', from=1-1, to=2-1]
    \end{tikzcd}\]
By the induction hypothesis, $\H^0\DB_X^p\to \DB_X^p$ are isomorphisms for all $p\le k-1$. A simple induction argument shows $\Omega_{X,h}^{\le k-1}\to \DB_X^{\le k-1}$ is an isomorphism, which implies the isomorphisms for the left-most and right-most columns. Thus, to show $f_i$ is an isomorphism for all $i$, it suffices to show that $\phi_i$ is an isomorphism for all $i$.

Now consider the following diagram involving local cohomology
\[    \begin{tikzcd} 
	{\HH^{i-1}(U,\Omega_{U,h}^{\le k})} & {\HH^i_{\Sigma_X}(X,\Omega_{X,h}^{\le k})} & {\HH^i(X,\Omega_{X,h}^{\le k})} & {\HH^i(U,\Omega_{U,h}^{\le k})} \\
	{\HH^{i-1}(U,\DB^{\le k}_U) } & {\HH^i_{\Sigma_X}(X,\DB^{\le k}_X) } & {\HH^i(X,\DB^{\le k}_X) } & {\HH^i(U,\DB^{\le k}_U) }.
	\arrow[from=1-2, to=1-3]
	\arrow[from=1-3, to=1-4]
	\arrow["{\psi_i}", from=1-2, to=2-2]
	\arrow[from=2-2, to=2-3]
	\arrow[from=2-3, to=2-4]
	\arrow["{\phi_i}", from=1-3, to=2-3]
	\arrow["{\cong }"', from=1-4, to=2-4]
	\arrow[from=2-1, to=2-2]
	\arrow[from=1-1, to=1-2]
	\arrow["\cong"', from=1-1, to=2-1]
    \end{tikzcd}\]
 The left-most and right-most maps are isomorphisms because $\H^0\DB^k_X \to \DB^k_X$ is an isomorphism on $U$. Thus, we reduce to showing that $\psi_i$ is an isomorphism for all $i$.

 \vspace{\medskipamount}

\noindent\textbf{Step 3.} We show that $\psi_i$ is injective for all $i$.

Since $X$ is normal and has pre-$k$-rational singularities, it has rational singularities and so, by Proposition \ref{lemma:h0-dual-and-log-differentials} and Remark \ref{remark:comparing-various-differentials}, we have the isomorphism  
    \[\H^0\DB^k_X\cong f_*\Omega^k_{\tilde X}(\log E) = \H^0(\DD_X(\DB_X^{n-k})).\] Moreover, the composition
    \[\H^0\DB^k_X \to \DB^k_X \to \DD_X(\DB_X^{n-k})\]
    is an isomorphism. Taking hypercohomology shows
    \[\HH^i_{\Sigma_X}(X, \H^0 \DB^k_X) \to \HH^i_{\Sigma_X}(X, \DB^k_X)\] 
    is injective for all $i$. Now consider the diagram
    \[    \begin{tikzcd} 
	{\HH^{i-1}_{\Sigma_X}(X,\Omega_{X,h}^{\le k-1})) } & {\HH^i_{\Sigma_X}(X,\H^0(\DB^k_X)[-k]) } & {\HH^i_{\Sigma_X}(X,\Omega_{X,h}^{\le k}))} & {\HH^i_{\Sigma_X}(X,\Omega_{X,h}^{\le k-1})) } \\
	{\HH^{i-1}_{\Sigma_X}(X, \DB^{\le k-1}_X) } & {\HH^i_{\Sigma_X}(X,\DB^k_X[-k])} & {\HH^i_{\Sigma_X}(X,\DB^{\le k}_X) } & {\HH^i_{\Sigma_X}(X, \DB^{\le k-1}_X) }.
	\arrow[from=1-2, to=1-3]
	\arrow[from=1-3, to=1-4]
	\arrow[hook, from=1-2, to=2-2]
	\arrow[from=2-2, to=2-3]
	\arrow[from=2-3, to=2-4]
	\arrow["{\psi_i}", from=1-3, to=2-3]
	\arrow["{\cong }"', from=1-4, to=2-4]
	\arrow[from=2-1, to=2-2]
	\arrow[from=1-1, to=1-2]
	\arrow["\cong"', from=1-1, to=2-1]
    \end{tikzcd}\]
Just as in Step 2, we have by induction that $\Omega_{X,h}^{\le k-1}\to \DB_X^{\le k-1}$ is an isomorphism, which implies the isomorphisms for the left-most and right-most columns. Now by a simple diagram chase, we conclude that $\psi_i$ is injective for all $i$.

\vspace{\medskipamount}

\noindent\textbf{Step 4.} We finish the proof by showing surjectivity of $\psi_i$ for all $i$. 

For this purpose, let $Y$ be a compactification of $X$ with boundary $D = Y \smallsetminus X$. Let $Z = D \bigsqcup \Sigma_X$, which is a closed set in $Y$. Let $U =Y\smallsetminus Z= X \smallsetminus \Sigma_X$, which is the locus where $\H^0\DB^k_X \to \DB^k_X$ is an isomorphism. 
    Since $Y$ is projective, by Proposition \ref{prop:Hodge-theory-surjection},
    \[\HH^i(Y, \Omega_{Y,h}^{\le k}) \twoheadrightarrow \HH^i(Y, \DB^{\le k}_Y)\]
    is surjective for all $i$. In fact, a similar map for local cohomology is also surjective. Indeed, consider the following commutative diagram, whose rows are given by long exact sequences of local cohomology:
    \[\begin{tikzcd} 
	{\HH^{i-1}(U,\Omega_{U,h}^{\le k})} & {\HH^i_Z(Y,\Omega_{Y,h}^{\le k})} & {\HH^i(Y,\Omega_{Y,h}^{\le k})} & {\HH^i(U,\Omega_{U,h}^{\le k})} \\
	{\HH^{i-1}(U,\DB^{\le k}_U) } & {\HH^i_Z(Y,\DB^{\le k}_Y) } & {\HH^i(Y,\DB^{\le k}_Y) } & {\HH^i(U,\DB^{\le k}_U) }
	\arrow[from=1-2, to=1-3]
	\arrow[from=1-3, to=1-4]
	\arrow[from=1-2, to=2-2]
	\arrow[from=2-2, to=2-3]
	\arrow[from=2-3, to=2-4]
	\arrow[two heads, from=1-3, to=2-3]
	\arrow["{\cong }"', from=1-4, to=2-4]
	\arrow[from=2-1, to=2-2]
	\arrow[from=1-1, to=1-2]
	\arrow["\cong"', from=1-1, to=2-1].
    \end{tikzcd}\]
    Diagram chasing shows that the morphism
    \[\HH^i_Z(Y, \Omega_{Y,h}^{\le k}) \twoheadrightarrow \HH^i_Z(Y, \DB^{\le k}_Y)\] 
    is surjective for all $i$. Moreover, this surjection is compatible with the decomposition of local cohomology $\HH^i_Z(Y,\blank) = \HH^i_{\Sigma_X}(Y,\blank) \oplus \HH^i_D(Y,\blank)$. Thus 
    \[\HH^i_{\Sigma_X}(Y, \Omega_{Y,h}^{\le k}) \twoheadrightarrow \HH^i_{\Sigma_X}(Y, \DB^{\le k}_Y)\]
    is surjective for all $i$. By excision, we conclude that the morphism
    \[\HH^i_{\Sigma_X}(X, \Omega_{X,h}^{\le k}) \xrightarrow{\psi_i} \HH^i_{\Sigma_X}(X, \DB^{\le k}_X)\]
    is surjective for all $i$. This completes the proof.
\end{proof} 


\begin{proof}[Proof of Corollary \ref{corollary:k-rational-implies-k-DB}]
    By Theorem \ref{theorem:pre-k-rational-implies-pre-k-DB}, $X$ has pre-$k$-Du Bois singularities. Thus we only need to show that $\Omega^p_X \to \H^0\DB^p_X$ is an isomorphism for all $0\le p\le k$. Since $X$ has rational singularities, we know that the natural map $\Omega^p_X \to \H^0(\DD_X(\DB^{n-p}_X))$ is an isomorphism for all $0 \leq p \leq k$. Thus by Remark \ref{remark:comparing-various-differentials}, we have that $\Omega^p_X \to \H^0\DB^p_X$ is an isomorphism for all $0\le p\le k$, as required.
\end{proof}

As another consequence of Theorem \ref{theorem:pre-k-rational-implies-pre-k-DB}, we extend the results concerning the symmetry of the Hodge-Du Bois numbers in \cite[Corollary 3.21]{FL-Saito} and \cite[Theorem 3.23]{FL-Saito}.
\begin{corollary}\label{cor:Hodge-symmetry}
    If $X$ is a normal proper variety with pre-$k$-rational singularities, then 
        \[\underline{h}^{p,q} = \underline{h}^{q,p} = \underline{h}^{n-p,n-q}\]
        for any $0\le p \le k$ and $0\le q\le n$, where $\underline{h}^{p,q}\coloneqq \dim  \HH^q(Y,\DB_X^p)$.
\end{corollary}
\begin{proof}
    The same proofs as in \cite[Corollary 3.21]{FL-Saito} and \cite[Theorem 3.23]{FL-Saito} work, provided one can show
    \[\psi_p:\DB^p_X \to \DD_X (\DB^{n-p}_X)\]
    is a quasi-isomorphism, for $p\le k$. Indeed, $\psi_p$ induces an isomorphism on $\H^0$ by Lemma \ref{lemma:h0-dual-and-log-differentials} and Remark \ref{remark:comparing-various-differentials}. For $i>0$, $\H^i$ of both sides vanish since $X$ is normal and pre-$k$-rational, hence pre-$k$-Du Bois by Theorem \ref{theorem:pre-k-rational-implies-pre-k-DB}. It follows that $\psi_p$ is a quasi-isomorphism for all $p\le k$.
\end{proof}

The proof of Theorem \ref{theorem:pre-k-rational-implies-pre-k-DB} also gives us the following result analogous to \cite[Corollary 2.5]{kovacs-rational-singularities-are-DB}.

\begin{proposition}\label{proposition-boutot-pre-k-DB}
    Let $\pi:Y \to X$ be a finite dominant morphism of normal varieties.
    \begin{enumerate}
        \item If $Y$ has rational singularities and pre-$k$-Du Bois singularities, then $X$ also has pre-$k$-Du Bois singularities.
        \item If $Y$ is smooth, then $X$ has pre-$k$-rational singularities for all $k \geq 0$.
    \end{enumerate}
\end{proposition}
\begin{proof}
(1) Observe first that $X$ has rational singularities by \cite[Theorem 1]{kovacs-characterization-of-rational-singularities}. We claim that for each $p \leq k$, there exists a map $\varphi:\DB_X^p \to \H^0\DB_X^p$ such that the composition
    \[\H^0\DB_X^p\to \DB_X^p\xto{\varphi} \H^0\DB_X^p\]
    is an isomorphism. Then the same proof as in Theorem \ref{theorem:pre-k-rational-implies-pre-k-DB}, with $\DD_X(\DB_X^{n-p})$ replaced by $\H^0\DB_X^p$, shows that $X$ has pre-$k$-Du Bois singularities.
    

    Consider the natural maps for each $p \leq k$:
    \[ \H^0\DB^p_X \to \DB^p_X \xrightarrow{\alpha} R\pi_*\DB^p_Y \simeq \pi_*\H^0\DB^p_Y \]
    where the isomorphism comes from the fact that $\pi$ is finite and $Y$ has pre-$k$-Du Bois singularities. Now since $X$ and $Y$ have rational singularities, by Remark \ref{remark:comparing-various-differentials} we have natural isomorphisms:
    \[ \H^0\DB^p_X \simeq (\Omega^p_X)^{\vee\vee}, \quad \H^0\DB^p_Y \simeq (\Omega^p_Y)^{\vee\vee}. \]
    Observe that we can find open subsets $U \subset X$ and $V \subset Y$ such that $U, V$ are smooth, $\pi^{-1}(U) = V$, and the complements $X \setminus U$, $Y \setminus V$ are of codimension $\geq 2$. We get the following commutative diagram:
    \[
    \begin{tikzcd}
        V \ar[r, hook, "i"] \ar[d, swap, "\pi'"] & Y \ar[d, "\pi"]\\
        U \ar[r, hook, "j"] & X.
    \end{tikzcd}
    \]
    By \cite{zannier-traces-of-differential-forms} applied to the finite map $\pi' : V \to U$, we have a trace map on differential forms:
    \[ \pi'_*\Omega^p_{V} \xrightarrow{\psi} \Omega^p_{U}, \]
    which is a left inverse to the natural inclusion. Therefore we get
    \begin{align*}
        \pi_*\left((\Omega^p_Y)^{\vee\vee} \right) \simeq \pi_*i_*\Omega^p_{V} \simeq j_*\pi'_*\Omega^p_{V} \xrightarrow{j_*\psi} j_*\Omega^p_{U} = (\Omega^p_X)^{\vee\vee}.
    \end{align*}
    Composing this map with $\alpha$ gives the desired $\varphi:\DB_X^p \to \H^0\DB_X^p$. 

(2) By the previous part, we have that $X$ has rational singularities and pre-$k$-Du Bois singularities for all $k \geq 0$. Hence 
    \[\DB_X^k \cong (\Omega_X^k)^{\vee \vee}\]
    for all $k \geq 0$. Consider the natural maps
    \[(\Omega_X^k)^{\vee \vee} \to \pi_*(\Omega_Y^k)^{\vee \vee}\] 
    with a left-inverse defined by the trace map as above. This implies that $(\Omega_X^k)^{\vee\vee}$ is a direct summand of $\pi_*(\Omega_Y^k)^{\vee \vee}$. Since $Y$ is smooth, $(\Omega_Y^k)^{\vee\vee}$ is a Cohen-Macaulay sheaf; by \cite[Proposition 5.4]{kollar-bir-geom}, we get $\pi_*(\Omega_Y^k)^{\vee \vee}$ is also Cohen-Macaulay. Thus $(\Omega_X^k)^{\vee\vee}$, as a direct summand of $\pi_*(\Omega_Y^k)^{\vee \vee}$, is Cohen-Macaulay for all $k \geq 0$. This implies 
    \[\H^i(\DD_X(\DB_X^{n-k})) \cong \sExt^{i-n}_{\O_X}((\Omega_X^{n-k})^{\vee \vee}, \omega_X^{\bullet}) = 0\]
    for $i>0$, by \cite[\href{https://stacks.math.columbia.edu/tag/0A7U}{Tag 0A7U}]{stacks-project}.
\end{proof}

As a corollary, we recover the following classical result (\cite[Theorem 5.3]{DB}, \cite[Corollary 2.47]{peters-steenbrink}) about finite quotient singularities.
\begin{corollary}\label{cor:quotient-singularities-pre-n-DB}
    Let $Y$ be a smooth affine variety and let $G$ be a finite group acting on $Y$. Then the quotient $X:= Y/G$ has pre-$k$-rational singularities (hence pre-$k$-Du Bois singularities) for all $k \geq 0$.
\end{corollary}

\begin{remark}\label{remark: equiv-def-pre-k-rat}
    Another possible definition for pre-$k$-rational-singularities is that 
        \[R^if_*\Omega^p_{\tilde X}(\log E) = 0\]
        for all $i>0$ and $0\le p\le k$; where $f:\tilde X \to X$ is any strong log resolution of singularities of $X$ with reduced exceptional divisor $E$. By Lemma \ref{lemma:h0-dual-and-log-differentials}(1), the two definitions agree if $k < \codim_X (X_{\sing})$. With this definition, stability under general hyperplane sections and the implication pre-$k$-rational $\implies$ pre-$k$-Du Bois continue to hold with similar proofs, using the exact sequence
        \[0\to \Omega^{p-1}_H(-H)\to \Omega_X^p(\log E)|_H \to \Omega_H^p(\log F)\to 0\]
        (where $F = E\cap H$) in place of the dual of the conormal sequence of Lemma \ref{lemma:conormal-DB-SES}. 
        However, the two definitions of pre-$k$-rational need not agree; see Remark \ref{remark:two-defn-k-rational-distinct} below. We choose to work with Definition \ref{definition:pre-k-DB-rational} because of the partial Hodge-Du Bois symmetry result in Corollary \ref{cor:Hodge-symmetry}.
\end{remark}

\section{New notions of $k$-Du Bois and $k$-rational singularities}\label{section:definitions}

Due to the non-reflexivity of the K\"ahler differentials of non-lci varieties, the definitions of strict $k$-Du Bois and strict $k$-rational singularities seem too restrictive to include interesting examples outside the lci case. As an explicit example, we have

\begin{example} \label{ex: cone-over-Pr}
The cone $Z$ over the $d$-th Veronese embedding of $\PP^r$ is not strict $1$-Du Bois, unless $d=1$, in which case $Z \cong \PP^{r+1}$ is smooth. Indeed, since  $Z$ is rational, if it has strict $1$-Du Bois singularities, then the K\"ahler differential $\Omega_Z^1 = \H^0\DB_X^1$ is reflexive by Remark \ref{remark:comparing-various-differentials}. However, this is not the case by \cite[Proposition 10]{greb-rollenske-torsion-kahler}, unless $d=1$.
Moreover, since $Z$ is a toric variety, it has finite quotient singularities, and is log terminal (even terminal when $r+1 >d$), which are mild singularities from the perspective of the Minimal Model Program.
\end{example}


In this section we take a closer look at the new definitions of $k$-Du Bois and $k$-rational singularities that we have proposed in this paper.
These definitions agree with the conditions (\ref{def:old-k-DB}) and (\ref{def:old-k-rat}) in the lci case, but are less restrictive in general, and in particular include examples like cones over Veronese varieties. 

\begin{definition} \label{def: new k-Du Bois}
Let $X$ be a variety. $X$ has \textbf{$k$-Du Bois singularities} if
\begin{enumerate}
\item $X$ is seminormal;
\item \label{condition: db-codim} $\codim_X (X_{\sing}) \ge 2k+1$;
\item \label{condition: db-pre} $X$ has pre-$k$-Du Bois singularities. 
\item \label{condition: db-ref} $\H^0\DB_X^p$ is reflexive, for all $p \le k$.
\end{enumerate}
\end{definition}

\begin{definition} \label{def: new k-rational}
Let $X$ be a variety. $X$ is said to have \textbf{$k$-rational singularities} if 
\begin{enumerate}
\item $X$ is normal;
\item \label{condition: rat-codim} $\codim_X (X_{\sing}) > 2k+1$;
\item \label{condition: rati-pre} For all $i>0$ and $0\le p\le k$, 
\[\bR^if_*\Omega^p_X(\log E) = 0\]
for any strong log resolution $f:\tilde X \to X$. 
\end{enumerate}
\end{definition}

\begin{remark} \label{remark:new-notions}
\begin{enumerate}
    \item Lemma \ref{lemma:h0-dual-and-log-differentials}(1) shows that if $p < \codim_X (X_{\sing})$, then 
    \[\bR f_*\Omega^p_{\tilde X}(\log E) \cong \DD_X(\DB^{n-p}_X).\]
    Thus condition (\ref{condition: rati-pre}) in Definition \ref{def: new k-rational} can be replaced by (\ref{condition: rati-pre}') that $X$ has pre-$k$-rational singularities. 
    \item There is no analogous condition (\ref{condition: db-ref}) in the definition of $k$-rational singularities. This is because if $X$ has $k$-rational singularities, it has rational singularities by Proposition \ref{proposition:k-DB-rational-0} below, and so by Remark \ref{remark:comparing-various-differentials}, $f_*\Omega^p_{\tilde{X}}(\log E)$ is reflexive.
    \item Conditions (\ref{condition: db-codim}) in both definitions are important properties enjoyed by $k$-Du Bois and $k$-rational singularities in the lci case (see \cite[Corollary 9.26]{MP-local-cohomology} and \cite[Corollary 1.3]{CDM-k-rational}), and they take presence in various arguments involving these singularities, most notably in the characterization of $k$-Du Bois and $k$-rational singularities in terms of the minimal exponents (\cite{MOPW}, \cite{JKSY}, \cite{MP-lci}, \cite{CDM-k-rational}), and in the proof of the local freeness theorem \cite[Theorem 1.2]{FL-Saito}. The mere vanishing of higher cohomology groups as imposed by pre-$k$-Du Bois or pre-$k$-rational singularities is not expected to ensure the validity of these deeper aspects of the theory of $k$-Du Bois and $k$-rational singularities. We hope that with the addition of the codimension bound condition (\ref{condition: db-codim}), one can eventually extend the local freeness theorem, among other results, to cases beyond local complete intersections. 
\end{enumerate}
\end{remark}

When $k=0$, we recover the usual notions of rational and Du Bois singularities.
\begin{proposition}\label{proposition:k-DB-rational-0}
A variety $X$ has $0$-Du Bois singularities if and only if it has Du Bois singularities. Similarly, it has $0$-rational singularities if and only if it has rational singularities.
\begin{proof}
If $X$ is Du Bois, then $\DB_X^0$ is quasi-isomorphic to $\O_X$, so conditions (\ref{condition: db-pre}) and (\ref{condition: db-ref}) are satisfied, because $\H^0\DB_X^0=\O_X$ is reflexive. By \cite[Corollary 0.3]{Saito-MHC}, $X$ is seminormal. Moreover, condition (\ref{condition: db-codim}) is vacuous when $k=0$. It follows that $X$ is $0$-Du Bois. Conversely, if $X$ is $0$-Du Bois, seminormality of $X$ implies $\H^0\DB_X^0 = \O_X$, by \cite[Corollary 0.3]{Saito-MHC}. This, together with condition (\ref{condition: db-pre}), implies $X$ is Du Bois. 

Note that by definition, a variety has rational singularities if and only if it is pre-$0$-rational and normal. Since $\codim_X (X_{\sing}) \ge 2$ when $X$ is normal, having $0$-rational singularities is equivalent to having rational singularities.
\end{proof}
\end{proposition}

When $X$ is lci, our new notions agree with the well-studied conditions (\ref{def:old-k-DB}) and (\ref{def:old-k-rat}).

\begin{proposition} \label{proposition: def-equiv-lci}
Let $X$ be a local complete intersection. Then $X$ has $k$-Du Bois singularities if and only if $X$ has strict $k$-Du Bois singularities. Similarly,  $X$ has $k$-rational singularities if and only if $X$ has strict $k$-rational singularities.
\begin{proof}
First, assume $X$ has strict $k$-Du Bois singularities. Clearly, $X$ is seminormal, and condition (\ref{condition: db-pre}) is satisfied. By \cite[Corollary 9.26]{MP-local-cohomology}, the codimension bound (\ref{condition: db-codim}) is satisfied. Since $X$ is lci, by \cite[Corollary 3.1]{MV-kahler-differential-lci}, we have $\H^0\DB_X^p\cong \Omega_X^p$ is reflexive when $p\le \codim_X (X_{\sing})-2$; hence condition (\ref{condition: db-ref}) follows from (\ref{condition: db-codim}). 
Conversely, assume $X$ has $k$-Du Bois singularities. Then \cite[Corollary 3.1]{MV-kahler-differential-lci} and the codimension bound (\ref{condition: db-codim}) implies
$\Omega_X^p$ is reflexive, for all $p\le k$. It follows that $X$ has strict $k$-Du Bois singularities.

Next, we assume $X$ has strict $k$-rational singularities. Clearly, $X$ is normal and pre-$k$-rational. By \cite[Corollary 1.3]{CDM-k-rational}, $X$ also satisfies condition (\ref{condition: rat-codim}).
Conversely, if $X$ has $k$-rational singularities, we just need to prove that $\Omega^p_X \cong f_*\Omega^p_{\tilde{X}}(\log E)$. Because of condition (\ref{condition: rat-codim}), we have $\Omega_X^p\cong \Omega_X^{[p]}$ by \cite[Corollary 3.1]{MV-kahler-differential-lci} and by Remark \ref{remark:comparing-various-differentials}, we have $f_*\Omega^p_{\tilde{X}}(\log E) \cong \Omega_X^{[p]}$. So $X$ has strict $k$-rational singularities.
\end{proof}
\end{proposition}

\begin{remark} \label{remark: comparision-of-two-definitions}
While the two notions agree when $X$ is lci, we do not know whether one implies the other in general.
\end{remark}

With these new definitions, the analogues of Theorems \ref{theorem:gen-hyperplane-sections-pre-k-DB} and \ref{theorem:pre-k-rational-implies-pre-k-DB} continues to hold, with some extra mild assumptions on $X$.
\begin{corollary}[=Theorem \ref{corollary:new-k-rat-implies-new-k-DB}]
Let $X$ be a quasi-projective variety. We have
\begin{enumerate}[label = (\alph*)]
    \item If $X$ has rational and $k$-Du Bois (resp. $k$-rational) singularities, then so does a general hyperplane section $H$ of $X$.
    \item If $X$ has $k$-rational singularities, then it has $k$-Du Bois singularities.
\end{enumerate}
\end{corollary}

\begin{proof}
First, we prove part $(a)$. In both statements, conditions (\ref{condition: db-codim}) follows because $H_{\sing}\subset X_{\sing}\cap H$, and condition (\ref{condition: db-pre}) is immediate from Theorem \ref{theorem:gen-hyperplane-sections-pre-k-DB}. It suffices to show if $X$ has $k$-Du Bois singularities, then a general hyperplane $H$ of $X$ satisfies condition (\ref{condition: db-ref}). Since $X$ has rational singularities, $H$ also has rational singularities. Thus Remark \ref{remark:comparing-various-differentials} implies that condition (\ref{condition: db-ref}) is satisfied for $H$ as well.

Now, we are left with part $(b)$. Condition (\ref{condition: rat-codim}) is clear, and condition (\ref{condition: rati-pre}) follows from Theorem \ref{theorem:pre-k-rational-implies-pre-k-DB}. Since $X$ has rational singularities, Remark \ref{remark:comparing-various-differentials} implies $\H^0\DB_X^p$ is reflexive for all $p$. Thus $X$ has $k$-Du Bois singularities.
\end{proof}

Our new definitions have the advantage of including large classes of non-lci examples.
\begin{example}
Toric varieties (\cite[Theorem 4.6]{GNPP}) and quotient singularities (Corollary \ref{cor:quotient-singularities-pre-n-DB}) have $k$-Du Bois singularities for $k\le(\codim_X (X_{\sing})-1)/2$. 

Similarly, by Proposition \ref{criteria:toric-varieties}, simplicial toric varieties have $k$-rational singularities for $k<(\codim_X(X_{\sing})-1)/2$. 
On the other hand, again by Proposition \ref{criteria:toric-varieties}, non-simplicial toric varieties have rational singularities but do not have $1$-rational singularities.
\end{example}

Criteria for cones over smooth projective varieties to have $k$-Du Bois and $k$-rational singularities can be given as follows. We will study cones extensively in Section \ref{section:cones}. In particular, 
it follows from Corollary \ref{cor:reflexive-push-forward} and Proposition \ref{corollary: pre-k-Du Bois for cone} that

\begin{proposition}[$k$-Du Bois criterion for cones] \label{prop: new-k-Du Bois for cone}
Let $X$ be a smooth projective variety and $L$ an ample line bundle on $X$. Let $Z=C(X,L)$ be the cone over $X$ with conormal bundle $L$. Then for $k\ge 1$, we have $Z$ has $k$-Du Bois singularities if and only if 
\begin{enumerate}
\item $k\le\frac{\dim X}{2}$;
\item $H^p(X,\O_X)=0$ for all $1\le p\le k$;
\item $H^i(\Omega_X^p\otimes L^m)=0$ for all $i\ge 1, m\ge 1$, and $0\le p\le k$.
\end{enumerate}
\end{proposition}

Similarly, from Proposition \ref{corollary: pre-k-rational for cone}, we obtain
\begin{proposition}[$k$-rational criterion for cones]\label{prop: new-k-raitonal for cone}
Let $X$ be a smooth projective variety and $L$ an ample line bundle on $X$ with conormal bundle $L$. Let $Z=C(X,L)$ be the cone over $X$. Then $Z$ is $k$-rational if and only if  
\begin{enumerate}
\item $k<\frac{\dim X}{2}$;
\item $H^0(X, \O_X) \xto{\cup c_1(L)} H^1(X,\Omega^1_X)\xto{\cup c_1(L)} H^2(X,\Omega_X^2)\to\cdots \to H^k(X,\Omega^k_X)$ are isomorphisms.
\item $H^i(X,\Omega_X^p\otimes L^m)=0$ for all $i\ge 1$, $m\ge 0$ and $0\le p\le k$, except possibly when $m=0$ and $i=p$.
\end{enumerate}
\end{proposition}

From Examples \ref{ex:cone-over-P^r-pre-DB}, \ref{ex: cone-over-Pr} ,\ref{ex:cone-over-K3-not-rational}, \ref{ex:cone-over-K3-pre-DB}, we immediately obtain the following concrete examples of varieties with $k$-rational and $k$-Du Bois singularities.
\begin{example}\label{ex:cone-over-Veronese-and-K3}
\begin{enumerate}
\item The cone $Z=C(\PP^r,\O(d))$ has $k$-Du Bois singularities if and only if $k\le\frac{r}{2}$. It is $k$-rational if and only if $k<\frac{r}{2}$.
\item 
The cone $Z=C(X,\O(l))$ over a quartic surface $X\subset \PP^3$ is not $2$-Du Bois, and is $1$-Du Bois if and only if $l\ge 5$. However, $Z$ does not have rational singularities.
\end{enumerate}
\end{example}

\begin{remark}\label{remark:-k-Du-Bois-but-not-k-1-rational}
When $X$ is lci, it is shown in \cite{FL-isolated} and \cite{CDM-k-rational} that $k$-Du Bois implies $(k-1)$-rational. Unfortunately, this fails to hold in general, as shown by the the examples of non-simplicial toric varieties and cones $C(X,\O(l))$ over quartic surface, where $l\ge 5$.
\end{remark}

\section{Toric Varieties}\label{section:toric-varieties}

In this section, we prove Proposition \ref{criteria:toric-varieties}.

Let $N$ be a free abelian group of rank $n$. Given a strongly convex, rational polyhedral cone $\sigma \subset N \otimes \mathbb{R}$, we have the associated $n$ dimensional affine toric variety $X :=X_\sigma$. For general notions regarding toric varieties, we refer to \cite{fulton:toric-varieties} and \cite{cox-little-schenk:toric-varieties}.

\subsection{Pre-$k$-Du Bois singularities for toric varieties}
By \cite[Theorem 4.6]{GNPP}, we have that the Du Bois complex of $X$ is quasi isomorphic to
\[0 \to \Omega^{[1]}_X \to \Omega^{[2]}_X\to \dots \to \Omega^{[n]}_X \to 0 \]
which is the complex of reflexive differentials of $X$ along with the trivial filtration. Therefore we have
\[ \DB^p_X \simeq \Omega^{[p]}_X. \]
Thus $X$ has pre-$k$-Du Bois singularities for all $0 \leq k \leq n$.

\subsection{Pre-$k$-rational singularities for toric varieties}

If $\sigma$ is simplicial, a result of Danilov (see \cite[Proposition 3.10]{Oda}) says that the natural map:
\[ \Omega^{[p]}_X \to \RHom_{\O_X}(\Omega^{[n-p]}_X,\Omega^{[n]}_X) \]
is an isomorphism. Since $X$ has rational singularities, it is Cohen-Macaulay, therefore the dualizing sheaf is just $\Omega^{[n]}_X$. Thus we get
\[ \Omega^{[p]}_X \to \RHom_{\O_X}(\Omega^{[n-p]}_X,\Omega^{[n]}_X) \simeq \DD_X(\DB^{n-p}_X). \]
Therefore, we conclude that $X$ has pre-$k$-rational singularities for all $0 \leq k \leq n$.

Now we consider the case when $\sigma$ is non-simplicial. We have $\codim_X(X_{\sing}) \geq 2$ by normality of $X$. Therefore, for a strong log resolution of $f:\tilde X \to X$ with reduced exceptional divisor $E$, we have that:
\[ \DD_X(\DB^{n-1}_X) \simeq \bR f_*\Omega^1_{\tilde X}(\log E)  \]
by Lemma \ref{lemma:h0-dual-and-log-differentials}. But \cite[Theorem 1.1(2)]{toric-paper} says that $R^1f_*\Omega^1_{\tilde X}(\log E) \neq 0$ because $\sigma$ is non-simplicial. Therefore, $X$ is not pre-$1$-rational.

\begin{remark}\label{remark:two-defn-k-rational-distinct}
    By \cite[Theorem 1.1(1)]{toric-paper}, when $\sigma$ is simplicial we have:
    \[R^{c-1}f_*\Omega^c_{\tilde X}(\log E) \neq 0 \]
    where $c = \codim_X(X_{\sing})$. This shows that $\bR f_*\Omega^p_{\tilde X}(\log E)$ is distinct from $\DD_X(\DB^{n-p}_X)$ for simplicial toric varieties.
\end{remark}

\section{Cones over smooth projective varieties} \label{section:cones}
In this section we describe the Du Bois complex for cones over a smooth projective variety $X$ and give conditions for cones to have (pre)-$k$-Du Bois and (pre)-$k$-rational singularities. 

\subsection{Convention for cones}
Let $X\subset \PP^N$ be a projective variety and $L$ an ample line bundle on $X$. Following \cite[Section 3.1]{MMP-singularities}, we define the \textit{affine cone over $X$ with conormal bundle $L$} as the affine variety
\[C(X,L) = \Spec\bigoplus_{m\ge 0} H^0(X,L^m).\]
If $X$ is normal, then the cone $C(X,\O_X(1))$ is the normalization of the \textit{classical affine cone} 
\[C(X):=Z(f_1=\cdots =f_s=0)\subset \AA^{N+1},\] 
where $f_1,\cdots, f_s$ are defining equations for the projective variety $X\subset \PP^{N}$.

There are advantages for working with the normalized cones $C(X,L)$ rather than the classic ones. First, by varying the ample line bundle $L$, we get more flexibility and could potentially construct more interesting examples. Moreover, non-normal cones cannot be Du Bois by the following Lemma, so they are not relevant to the study of $k$-Du Bois and $k$-rational singularities.

\begin{lemma}
Let $X$ be a smooth variety and $C(X)$ the classical cone over $X$. If $C(X)$ is Du Bois, then it is normal. 
\begin{proof}
The normalization map $C(X,\O_X(1))\to C(X)$ is a seminormalization, since it induces a bijection on the underlying set of $C(X,\O_X(1))$ and $C(X)$. Since $C(X)$ is Du Bois, by \cite[Proposition 5.2]{Saito-MHC}, it is seminormal. Thus $C(X)=C(X,\O_X(1))$ is normal.
\end{proof}
\end{lemma}

\subsection{Du Bois complexes for cones}
We describe the Du Bois complexes of cones $Z=C(X,L)$ over a smooth projective variety $X$ in terms of cohomology of the K\"ahler differentials on $X$. 

\begin{proposition}\label{prop:DB-complex}
Let $X$ be a smooth projective variety and $L$ an ample line bundle on $X$. Then
\[\H^0\DB_Z^0=\O_Z \quad \text{ and }\]
\[\Gamma(\H^i\DB_Z^0) = \bigoplus_{m\ge 1} H^i(X,L^m) \quad \text{ for all } i>0.\]
For $p\ge 1$ and all $i\ge 0$,
\[\Gamma(\H^i\DB_Z^p)=\bigoplus_{m\ge 1} H^i(\Omega_X^p\otimes L^m)\oplus \bigoplus_{m\ge 1} H^i(\Omega_X^{p-1}\otimes L^m).\]
Note that since $Z$ is affine, the global sections $\Gamma(\H^i\DB_Z^p)$ determine $\H^i\DB_Z^p$ completely.
\end{proposition}
 
In the form of cdh-sheaves, our result is subsumed by \cite[Lemma 2.4, Corollary 2.10]{CHWW-cones}, which works over arbitrary fields of characteristic zero. For the convenience of the reader, we include a proof of Proposition \ref{prop:DB-complex} in Appendix \ref{appendix:cones}.

\vspace{\medskipamount}
    
As a corollary of the proof, we obtain a criterion for $\H^0\DB_X^p$ to be reflexive.

\begin{corollary}\label{cor:reflexive-push-forward}
Let $X$ be a smooth projective variety, and $L$ an ample line bundle on $X$. Let $Z=C(X,L)$ be the cone over $X$, and $f: \tilde{Z}\to Z$ be the blowup at the cone point. Then
\begin{enumerate}
    \item For $p\ge 1$, the natural map $\H^0\DB_Z^p\to f_*\Omega^p_{\tilde{Z}}$ is an isomorphism if and only if
\[H^0(X,\Omega_X^p)=0.\]
    \item For $1\le p\le \dim Z-2$, the sheaf $\H^0\DB_X^p$ is reflexive if and only if $H^0(X,\Omega_X^p)=0$.
\end{enumerate}
\end{corollary}

\begin{proof}
By Proposition \ref{prop:DB-complex}, for $p\ge 1$,
\[\Gamma(\H^0\DB_Z^p) \cong \bigoplus_{m\ge 1} H^0(\Omega_X^p\otimes L^m)\oplus \bigoplus_{m\ge 1}H^0(\Omega_X^{p-1}\otimes L^m).\]
By Lemma \ref{Lemma:SES-cones}, for $p\ge 1$,
\[\Gamma(f_*\Omega^p_{\tilde{Z}})\cong \bigoplus_{m\ge 0} H^0(X,\Omega_X^p\otimes L^m)\bigoplus_{m\ge 1} H^0(X,\Omega_X^{p-1}\otimes L^m).\]
Moreover, a close inspection of the proof of Proposition \ref{prop:DB-complex} shows that the natural map $\H^0\DB_Z^p\to f_*\Omega^p_{\tilde{Z}}$ induces the inclusion map of the right-hand sides as direct summands. A comparison between these two formulas gives (1).

(2) is a consequence of (1) and \cite[Theorem 1.3]{SSt85}, which implies that $f_*\Omega_{\tilde{Z}}^p \cong \Omega_Z^{[p]}$ for $p \le \dim Z-2$.
\end{proof}

\subsection{Criteria for cones to have $k$-Du Bois and $k$-rational singularities}
As an immediate consequence of Proposition \ref{prop:DB-complex}, we obtain 

\begin{corollary}[Pre-$k$-Du Bois criterion for cones]
\label{corollary: pre-k-Du Bois for cone}
Let $X$ be a smooth projective variety, and $L$ an ample line bundle on $X$. Let $Z=C(X,L)$ be the affine cone over $X$ with conormal bundle $L$. Then $Z$ has pre-$k$-Du Bois singularities if and only if 
\[H^i(X,\Omega_X^{p}\otimes L^m)=0\]
for all $i\ge 1, m\ge 1$ and $p\le k$.
\end{corollary}

We learned from B. Tighe (private communication) that similar criterion for the cone $C(X,L)$ to have strict $k$-Du Bois singularities will appear in his Ph.D. thesis.

\begin{remark}
Since $Z$ is normal,
it is pre-$0$-Du Bois if and only if it is Du Bois. Both conditions are equivalent to
\[H^i(X,L^m)=0, \text{ for all }i>0, m\ge 1\]
by the above corollary. This recovers \cite[Proposition 4.13]{DB} (see also \cite[Theorem 2.5]{graf-kovacs}).
\end{remark}

A smooth projective variety $X\subset \PP^N$ is said to \textit{satisfy Bott vanishing} if 
\[H^i(\Omega_X^k\otimes L)=0\]
for all $i>0$, $k\ge 0$ and any ample line bundle $L$ on $X$. From Corollary \ref{corollary: pre-k-Du Bois for cone}, we immediately obtain
\begin{corollary}\label{cor:Bott-van}
If a smooth projective variety $X$ satisfies Bott vanishing, and $L$ is any ample line bundle on $X$, then the cone $C(X,L)$ is pre-$k$-Du Bois for all $k$.
\end{corollary}

Bott vanishing has been studied extensively in the literature. For example\footnote{Some of these examples satisfy Bott vanishing over more general base fields. For simplicity, we state the results over the complex numbers.}, projective spaces \cite{Bott-proj-space}, toric varieties \cite{Danilov-toric,Batyrev-Cox-toric,char-p-toric}, abelian varieties, del Pezzo surfaces of degree $\ge 5$ \cite{Totaro-Bott-vanishing}, K3 surfaces of Picard number 1 and degree $20$ and $\ge 24$ \cite{Bott-van-K3-1,Bott-van-K3-2,Totaro-Bott-vanishing}, stable GIT quotients of $\PP^n$ by the action of $\PGL_2$ \cite{Torres-Bott-van-GIT}, $37$ Fano $3$-folds \cite{Totaro-Bott-van-Fano}, projective varieties with int-amplified endomorphism \cite{Kawakami-Totaro-Bott-van}, all satisfy Bott vanishing. It follows that the classical cones over any projectively normal embedding of these varieties are pre-$k$-Du Bois for all $k$.

Next, we state the criteria for the cone $Z=C(X,L)$ to be $k$-rational.

\begin{corollary}[Pre-$k$-rational criterion for cones] \label{corollary: pre-k-rational for cone}
Let $X$ be a smooth projective variety of dimension $n$ and $L$ an ample line bundle on $X$. Let $Z=C(X,L)$ be the cone over $X$ with conormal bundle $L$. Then
\begin{enumerate}
\item \label{condition: pre-k-rat-cones} For $k\le n$, the cone $Z$ has pre-$k$-rational singularities if and only if
\begin{itemize}
\item $H^i(X,\Omega_X^p\otimes L^m)=0$ for all $i\ge 1$, $m\ge 0$ and $0\le p\le k$, except possibly when $m=0$ and $i=p$.
\item $H^0(X, \O_X) \xto{\cup c_1(L)} H^1(X,\Omega^1_X)\xto{\cup c_1(L)} H^2(X,\Omega_X^2)\to\cdots \to H^k(X,\Omega^k_X)$ are isomorphisms.
\end{itemize}

\item If $Z$ has pre-n-rational singularities, it also has pre-$(n+1)$-rational singularities. 

\end{enumerate}
\end{corollary}

\begin{proof}
We prove (\ref{condition: pre-k-rat-cones}) by induction on $k$. For the base case $k=0$, by Lemma \ref{Lemma:SES-cones}, the cone $Z$ has pre-$0$-rational singularities if and only if
\[\Gamma(R^i f_*\O_{\tilde{Z}}) \cong \bigoplus_{m\ge 0} H^i(X,L^m)=0\]
if and only if 
\[H^i(X,L^m)=0\]
for all $i\ge 1$ and $m\ge 0$. Suppose we have the characterization of pre-$p$-rational singularities for $0\le p \le k-1$. Since $k < n+1 = \codim_Z (Z_{\sing})$, by Lemma \ref{lemma:h0-dual-and-log-differentials}, the cone $Z$ has pre-$k$-rational singularities if and only if 
\[\bR^i f_*\Omega_{\tilde Z}^p(\log X) = 0\]
for all $i>0$ and $0\le p\le k$.
This is equivalent to
\[\varphi_i: H^i(X,\Omega_X^{k-1})\to H^{i+1}(\tilde{Z},\Omega_{\tilde{Z}}^{k})\]
being isomorphisms for all $i>0$ and surjective for $i=0$,  by the long exact sequence associated to the residue sequence
\[0\to \Omega^k_{\tilde{Z}} \to \Omega^k_{\tilde{Z}}(\log X) \to \Omega^{k-1}_X \to 0.\] 
By taking cohomology of the split short exact sequence in Lemma \ref{Lemma:SES-cones}, we obtain
\[H^{i+1}(\tilde{Z},\Omega_{\tilde{Z}}^{k}) \cong  \bigoplus_{m\ge 0} H^{i+1}(X,\Omega_X^k\otimes L^m)\oplus\bigoplus_{m\ge 1} H^{i+1}(X,\Omega_X^{k-1}\otimes L^m)\]
for $i\ge 0$. Moreover, the morphism $\varphi_i$, as the connecting homomorphism of the residue sequence, is identified with the composition 
\[H^{i}(X,\Omega_X^{k-1})\xto{\cup c_1(L)} H^{i+1}(X,\Omega_X^k) \to H^{i+1}(\tilde Z, \Omega^k_{\tilde Z})\]
where $H^{i+1}(X,\Omega_X^k)$ is a direct summand of $H^{i+1}(\tilde{Z},\Omega_{\tilde{Z}}^k)$. 
Since $Z$ has pre-$(k-1)$-rational singularities by the induction hypothesis, we get $H^{i}(X,\Omega_X^{k-1})=0$ unless $i=k-1$. It then follows that $\varphi_i$ is an isomorphism if and only if
\[H^i(X,\Omega_X^k\otimes L^m)=0\]
for all $i\ge 1$ and $m\ge 0$, except possibly when $i=k$ and $m=0$, in which case
\[H^{k-1}(X,\Omega_X^{k-1})\xto{\cup c_1(L)} H^k(X,\Omega_X^k)\]
is an isomorphism. This completes the proof of (\ref{condition: pre-k-rat-cones}). 

For the second statement, if $Z$ has pre-$n$-rational singularities, we need to show that 
\[\H^i(\DD_Z(\DB_Z^0))=0\] 
for all $i>0$. Consider the following commutative diagram
\[\begin{tikzcd}
	{\CC[-n-1] \cong \DD_Z(\DB_{pt}^0)} & {\DD_Z(\DB_Z^0)} & {\bR f_*\omega_{\tilde Z}(X)} & { } \\
	{\bR f_*\omega_{X}[-1]} & {\bR f_*\omega_{\tilde Z}} & {\bR f_*\omega_{\tilde Z}(X)} & { }
	\arrow[from=1-2, to=1-3]
	\arrow[from=1-1, to=1-2]
	\arrow[from=2-1, to=1-1]
	\arrow[from=2-2, to=1-2]
	\arrow[from=2-1, to=2-2]
	\arrow[from=2-2, to=2-3]
	\arrow[equal, from=1-3, to=2-3]
	\arrow["{+1}", from=1-3, to=1-4]
	\arrow["{+1}", from=2-3, to=2-4]
\end{tikzcd}\]
By using Grauert-Riemanschneider vanishing theorem and taking cohomology, we obtain
\[\H^i(\DD_Z(\DB^0_Z)) \cong \bR^i f_*\omega_{\tilde Z}(X) \cong H^i(X, \omega_X)= 0\]
for $0< i< n$; and 
\[\begin{tikzcd}
	0 & {\H^{n} (\DD_Z(\DB_Z^0))} & {\bR^n f_*\omega_{\tilde Z}(X)} & \CC & {\H^{n+1} (\DD_Z(\DB_Z^0))} & 0 \\
	& 0 & {\bR^n f_*\omega_{\tilde Z}(X)} & {H^n(X, \omega_X)} & 0
	\arrow[from=1-5, to=1-6]
	\arrow[from=1-4, to=1-5]
	\arrow[from=2-4, to=2-5]
	\arrow[from=1-3, to=1-4]
	\arrow[from=1-2, to=1-3]
	\arrow[from=1-1, to=1-2]
	\arrow["\cong"', from=2-4, to=1-4]
	\arrow[from=2-5, to=1-5]
	\arrow[equal, from=2-3, to=1-3]
	\arrow["\cong", from=2-3, to=2-4]
	\arrow[from=2-2, to=2-3]
	\arrow[from=2-2, to=1-2]
\end{tikzcd}\]
The diagram induces the isomorphism 
\[\bR^n f_*\omega_{\tilde Z}(X) \xto{\sim} \CC,\]
which implies 
\[\H^{n} \DD_Z(\DB_Z^0) = \H^{n+1} \DD_Z(\DB_Z^0) = 0.\]

\end{proof}

\begin{remark}
Since $Z$ is normal, it has rational singularities if and only if it has pre-$0$-rational singularities if and only if
\[H^i(X,L^m)=0, \text{ for all }i>0, m\ge 0\]
by our criterion. This recovers the well-known \cite[Proposition 3.13]{MMP-singularities}.
\end{remark}


\subsection{Cones over Veronese varieties}
Following \cite{greb-rollenske-torsion-kahler}, we denote by $X_{d,r}$ the classical affine cone over the $d$-th Veronese varieties of $\PP^r$. Since the Veronese varieties are projectively normal, $X_{d,r}=C(\PP^r,\O_{\PP^r}(d))$ in our notation. 

\begin{example}\label{ex:cone-over-P^r-pre-DB}
$X_{d,r}$ has pre-$k$-rational singularities for all $k$. Indeed, Bott vanishing theorem for $\PP^r$ shows that 
\[H^i(\PP^r, \Omega^k_{\PP^r}(md))=0\]
for all $i>0, k\ge 0$ and $m\ge 0$, except when $m=0$ and $i=k$. In that case, the ismorphisms  
\[H^0(\PP^r, \O_{\PP^r})\overset{\simeq}{\longrightarrow}H^1(\PP^r, \Omega^1_{\PP^r})\overset{\simeq}{\longrightarrow}\cdots \overset{\simeq}{\longrightarrow} H^r(\PP^r, \omega_{\PP^r})\]
are given by wedging with the hyperplane class.

By theorem \ref{theorem:pre-k-rational-implies-pre-k-DB}, $X_{d,r}$ has pre-$k$-Du Bois singularities for all $k$. This is consistent with Corollary \ref{cor:Bott-van} and the fact $\PP^r$ satisfies Bott vanishing. Note also that $X_{d,r}$ is a toric variety, and this example is consistent with the results in section \ref{section:toric-varieties}.
\end{example}

\subsection{Cones over K3 surfaces}

Let $X$ be a smooth quartic in $\PP^3$, which is a $K3$ surface. Let $L=\O_X(l)$ be an ample line bundle on $X$. It is not hard to check that the image $X'\subset \PP^n$ of the embedding of $X$ by $L$ is projectively normal, so  $C(X,L)=C(X')$ is the classical cone over $X'$.

\begin{example}\label{ex:cone-over-K3-not-rational}
For any $l\ge 1$, the affine cone $Z=C(X,\O_X(l))$ is Du Bois but does not have rational singularities. Indeed, since $H^2(X,\O_X)\neq 0$, $Z$ never has rational singularities by Proposition \ref{corollary: pre-k-rational for cone}. On the other hand, $H^i(X,\O_X(m))=0$ for all $i>0$ and $m\ge 1$ by the Kodaira vanishing, so $Z$ is Du Bois for any $l\ge 1$.
\end{example}

\begin{example}\label{ex:cone-over-K3-pre-DB}
The affine cone $Z=C(X,\O_X(l))$ has pre-$1$-Du Bois singularities if and only if it has pre-$2$-Du Bois singularities, if and only if $l\ge 5$. By the Kodaira-Akizuki-Nakano vanishing, we get $H^2(X,\Omega_X^1(m))=0$ and $H^1(X,\Omega_X^2(m))=0$ for all $m\ge 1$. It remains to compute $H^1(X,\Omega_X^1(m))$. By considering the restriction of the Euler sequence to $X$
\[0\to \Omega_{\PP^3}^1|_X(m)\to \O_X(m-1)^{\oplus 4}\to \O_X(m)\to 0,\]
together with $X$ being projectively normal, we have $H^1(X, \Omega_{\PP^3}^1|_X(m))=0$ for $m\ge 1$, and $H^2(X, \Omega_{\PP^3}^1|_X(m))=0$ for $m\ge 2$. Moreover,  $H^2(\Omega_{\PP^3}^1|_X(1))\cong H^0(X,\O_X)^{\oplus 4}$ is a 4-dimensional vector space.
From the conormal exact sequence
\[0\to \O_X(m-4)\to \Omega_{\PP^3}^1|_X(m) \to \Omega_X^1(m)\to 0,\]
we obtain
\[H^1(\Omega_X^1(m))\cong \ker\bracket{H^2(\O_X(m-4))\to H^2(\Omega_{\PP^3}^1|_X(m))}.\]
By comparing dimensions of both sides, we see that $H^1(\Omega_X^1(m))=0$ if and only if $m\ge 5$.

\end{example}

\appendix
\section{Du Bois complex of cones over smooth projective varieties}\label{appendix:cones}

In this appendix, we give a proof to Proposition \ref{prop:DB-complex}.

\subsection{Computation of $\DB_Z^0$} Recall that $Z$ has a natural strong log resolution given by 
\[f: \tilde Z = {\bf Spec}_X\bigoplus_{m\ge 0} L^m \to Z\]
with the exceptional divisor $E\cong X$. Note that $\tilde Z$ is a vector bundle over $X$ via the projection $\pi:\tilde Z \to X$. From the blow-up square
\[\begin{tikzcd}
	E\cong X & {\tilde{Z}} \\
	v & Z
	\arrow["i", from=1-1, to=1-2]
	\arrow["f"', from=1-1, to=2-1]
	\arrow[from=2-1, to=2-2]
	\arrow["f", from=1-2, to=2-2]
\end{tikzcd}\]
where $v$ is the vertex of the cone $Z$, we obtain an exact triangle
\[\DB_Z^0\to \bR f_*\O_{\tilde{Z}}\oplus \O_{v} \to \bR f_*\O_X \xto{+1}.\]
Since the composition $X \xto{i} \tilde Z \xto{\pi} X$ is the identity, the associated long exact sequence of cohomology breaks into short exact sequences
\[0 \to \H^0\DB_Z^0 \to f_*\O_{\Tilde{Z}}\oplus \O_v \to f_*\O_X \to 0,\]
and for $i\ge 1$,
\[0 \to \H^i\DB_Z^0 \to R^if_*\O_{\Tilde{Z}} \to R^if_*\O_X \to 0.\]
Since $Z$ is affine, by taking global sections, we obtain short exact sequences
\[0\to \Gamma(\H^0\DB_Z^0)\to H^0(\tilde{Z},\O_{\tilde{Z}})\oplus \CC \to H^0(X,\O_X)\to 0,\]
and for $i\ge 1$,
\[0\to \Gamma(\H^i\DB_Z^0)\to H^i(\tilde{Z},\O_{\tilde{Z}}) \to H^i(X,\O_X)\to 0.\]
Note that the first map in the composition
\[\O_X\to \pi_*\O_{\tilde{Z}} \cong \bigoplus_{m\ge 0} L^m \to \O_X\]
is given by the inclusion of the $m=0$ summand. 
Thus the composition
\[H^i(X,\O_X)\to H^i(\tilde{Z},\O_{\tilde{Z}}) \cong \bigoplus_{m\ge 0} H^i(X,L^m) \to H^i(X,\O_X)\]
is an isomorphism for $i\ge 0$, in which
the first map is given by the inclusion of the $m=0$ summand.
Applying this to the above exact sequence, we obtain 
\[\Gamma(\H^i\DB_Z^0) \cong \bigoplus_{m\ge 1} H^i(X,L^m),\]
for any $i\ge 1$. 
For $i=0$, the map
\[H^0(\tilde{Z},\O_{\tilde{Z}})\oplus \CC \cong 
\bigoplus_{m\ge 0} H^0(X,L^m)\oplus \CC\to H^0(X,\O_X)\]
is given by $(x, \alpha) \mapsto \varphi(x)-\alpha$, where $\varphi: \bigoplus_{m\ge 0}H^i(X,L^m)\to H^0(X,\O_X)$ is the projection onto the $m=0$ summand. Its kernel is thus
\[\Gamma(\H^0\DB_Z^0)\cong \bigoplus_{m\ge 0} H^0(X,L^m)\cong \O_Z.\]

\subsection{Computation of $\DB_Z^p$ for $p\ge 1$.}
The strategy is similar, except that the isomorphism $\pi_*\O_{\tilde{Z}}\cong \bigoplus_{m\ge 1} L^m$ is replaced by the following short exact sequence.

\begin{lemma}\label{Lemma:SES-cones}
Let $X$ be a smooth projective variety and $L$ an ample line bundle on $X$. Then for each $p\ge 1$, there is a split exact sequence
\[0\to \bigoplus_{m\ge 0}\Omega_X^p\otimes L^m \to \pi_*\Omega^p_{\tilde{Z}}\to \bigoplus_{m\ge 1}\Omega_X^{p-1}\otimes L^m\to 0.\]
Thus for each $i\ge 0$ and $p\ge 1$, we have a split short exact sequence
\[0\to \bigoplus_{m\ge 0}H^i(X,\Omega_X^p\otimes L^m) \to H^i(\tilde{Z},\Omega^p_{\tilde{Z}})\to \bigoplus_{m\ge 1}H^i(X,\Omega_X^{p-1}\otimes L^m)\to 0.\]
\begin{proof}
Consider the sequence of differentials
\[0\to \pi^*\Omega_X^1\to \Omega^1_{\tilde{Z}}\to \Omega^1_{\tilde{Z}/X} \cong \pi^*L\to 0.\] 
Taking $p$-th exterior powers gives
\[0\to \pi^*\Omega_X^p\to \Omega_{\tilde{Z}}^p\to \pi^*(\Omega_X^{p-1}\otimes L)\to 0.\]
By pushing forward via the affine morphism $\pi: \tilde{Z}\to X$, we obtain
\[0\to \bigoplus_{m\ge 0}\Omega_X^p\otimes L^m\to \pi_*\Omega^p_{\tilde{Z}}\to \bigoplus_{m\ge 1}\Omega_X^{p-1}\otimes L^m\to 0.\]
The composition
\[\bigoplus_{m\ge 1}\Omega_X^{p-1}\otimes L^m\to \pi_*\Omega_{\tilde{Z}}^{p-1}\xto{d} \pi_*\Omega^{p}_{\tilde{Z}}\]
gives a splitting morphism.  
To see this, consider affine open set $U = \Spec A\subset X$, on which $L$ is trivial, say $L|_U = tA$. We have $\pi_*\O_{\tilde{Z}}|_U = A[t]$. Thus the splitting morphism is given by 
\[\omega\otimes t^m \xmapsto{d} d\omega\otimes t^m + m\omega t^{m-1} dt,\] 
where $\omega$ is any $(p-1)$-form on $X$. Note that $d\omega\otimes t^m\in \Omega_{A}^{p-1}\otimes A[t]$ is a local section of $\bigoplus_{m\ge 0} \Omega_A^p\otimes L^m$, hence is killed by the morphism $\pi_*\Omega^p_{\tilde Z} \to \bigoplus_{m\ge 1}\Omega_X^{p-1}\otimes L^m$. Since the composition map can be identified with multiplication by $m$, it is an isomorphism when we are in characteristic zero.
\end{proof}
\end{lemma}

Now we proceed as in the case $p=0$. The long exact sequence associated to the exact triangle
\[\DB_Z^p\to \bR f_*\Omega_{\tilde{Z}}^p \to \bR f_*\Omega_X^p \xto{+1}\]
splits into short exact sequences
\[0 \to \H^i\DB_Z^p \to R^if_*\Omega^p_{\Tilde{Z}} \to R^if_*\Omega_X^p \to 0,\]
which by taking global section, gives
\begin{equation*}
0\to \Gamma(\H^i\DB_Z^p)\to H^i(\tilde{Z},\Omega^p_{\tilde{Z}})\to H^i(X,\Omega^p_X)\to 0
\end{equation*}
for any $i\ge 0$ and $p\ge 1$. 
It then follows from Lemma \ref{Lemma:SES-cones} that
\[\Gamma(\H^i\DB_Z^p)\cong \bigoplus_{m\ge 1} H^i(\Omega_X^p\otimes L^m)\oplus \bigoplus_{m\ge 1} H^i(\Omega_X^{p-1}\otimes L^m)\]
for all $p\ge 1$ and $i\ge 0$. This completes the proof.

\bibliographystyle{alpha}
\bibliography{reference}

\Addresses
\end{document}